\documentclass[10pt,a4paper]{article}
\usepackage{amsfonts}
\usepackage[latin9]{inputenc}
\usepackage{amsmath}
\usepackage{amssymb}
\usepackage{breqn}
\usepackage{url}
\usepackage{graphicx}
\usepackage[pdfstartview=FitH]{hyperref}
\usepackage{amsthm}
\usepackage{authblk}
\usepackage{hyperref}
\usepackage{url}
\makeatletter
\begin{document}
\newtheorem{theorem}{Theorem}[section]
\newtheorem{lemma}[theorem]{Lemma}
\newtheorem{definition}[theorem]{Definition}
\newtheorem{claim}[theorem]{Claim}
\newtheorem{example}[theorem]{Example}
\newtheorem{remark}[theorem]{Remark}
\newtheorem{proposition}[theorem]{Proposition}
\newtheorem{corollary}[theorem]{Corollary}

\newcommand{\im}{\operatorname{Im}}
\newcommand{\Ker}{\operatorname{Ker}}
\newcommand{\Int}{\operatorname{Int}}

\title{Local spectral expansion approach to high dimensional expanders part I: Descent of spectral gaps}
\author{Izhar Oppenheim}
\affil{Department of Mathematics, Ben-Gurion University of the Negev, Be'er Sheva 84105, Israel, izharo@bgu.ac.il}

\maketitle

\begin{abstract}
This paper introduces the notion of local spectral expansion of a simplicial complex as a possible analogue of spectral expansion defined for graphs. We then show that the condition of local spectral expansion for a complex yields various spectral gaps in both the links of the complex and the global Laplacians of the complex.
\begin{flushleft}
\end{flushleft}
\end{abstract}

\section{Introduction}

Let  $G=(V,E)$ be a finite graph without loops or multiple edges. For a vertex $u \in V$, denote by $m(u)$ the valency of $u$, i.e.,
$$m(u) = \vert \lbrace (u,v) \in E \rbrace \vert . $$

Recall that the (normalized) Laplacian on $G$ is a positive operator $\mathcal{L}$ on $L^2 (V,\mathbb{R})$ defined by the matrix
$$\mathcal{L} (u,v)= \begin{cases}
1 & u=v \\
-\dfrac{1}{\sqrt{m(u)m(v)}} & (u,v) \in E \\
0 & \text{otherwise}
\end{cases}$$
If $G$ is connected then $\mathcal{L}$ has the eigenvalue $0$ with multiplicity $1$ (the eigenvector is the constant function) and all other the eigenvalues are positive. Denote by $\lambda (G)$ the smallest positive eigenvalue of $\mathcal{L}$ of $G$. $\lambda (G)$ is often referred to as the spectral gap of $G$. A family of graphs $\lbrace G_j = (V_j, E_j) \rbrace_{j \in \mathbb{N}}$ is a family of expanders if all the graphs $G_j$ are connected, $\vert V_j \vert \rightarrow^{j \rightarrow \infty} \infty$, and there is $\lambda >0$ such that for every $j$, $\lambda (G_j) \geq \lambda$.
\begin{remark}
In some sources, the definition of an expander family also include the condition that the valency of the graphs is uniformly bounded, i.e., that
 $\sup_j \sup_{v \in V_j} m(v) < \infty$.
\end{remark}

For some applications one is interested not just $\lambda (G)$ but also in the largest eigenvalue of $\mathcal{L}$, denoted here as $\kappa (G)$. For $\lambda >0, \kappa <2$, we shall call $G$ a two-sided $(\lambda, \kappa)$-expander if $\lambda (G) \geq \lambda$ and $\kappa (G) \leq \kappa$. 

In recent years, expanders had vast applications in pure and applied mathematics (see \cite{LubExpGr}). This fruitfulness of the theory of expander graph, raises the question - what should be the high dimensional analogue of expanders?, i.e., what is the analogous definition of an expander complex when one considers a $n$-dimensional simplicial complex, $X$, instead of a graph. In \cite{LubHighDim} two main approaches are suggested: The first is through the $\mathbb{F}_2$-coboundary expansion of $X$ originated in \cite{Grom}, \cite{LM} and \cite{MW} . The second is through studying the spectral gap of the $(n-1)$-Laplacian of $X$ (where $n$ is the dimension of $X$)  or the spectral gaps of all $0,..,(n-1)$-Laplacians of $X$ (see \cite{P}, \cite{PRT}). One of the difficulties with both approaches are that both the $\mathbb{F}_2$-coboundary expansion and the spectral gap of the $(n-1)$-Laplacian are usually hard to calculate or even bound in examples. 

This paper suggests a new approach that we call ``local spectral expansion'' (or $1$-dimensional spectral expansion). For a simplicial complex $X$, we denote $X^{(j)}$ to be the set of all $j$-simplices in $X$. Recall that for a simplicial complex $X$ of dimension $n$ and a simplex $\lbrace u_0,...,u_k \rbrace \in X^{(k)}$, the link of $\lbrace u_0,...,u_k \rbrace$ denoted $X_{\lbrace u_0,...,u_k \rbrace}$ is a simplicial complex of dimension $\leq n-k-1$ defined as:
$$X_{\lbrace u_0,...,u_k \rbrace}^{(j)} = \lbrace \lbrace v_0,...,v_j \rbrace \in X^{(j)} : \lbrace u_0,...,u_k,v_0,...,v_j \rbrace \in X^{(k+j+1)} \rbrace .$$
Note that if $X$ is pure $n$ dimensional (i.e., every simplex of $X$ is a face of a simplex of dimension $n$), then $X_{\lbrace u_0,...,u_k \rbrace}$ is of dimension exactly $n-k-1$. Next, we can turn to define local spectral expansion:
\begin{definition}
For $\lambda >\frac{n-1}{n}$, a pure $n$-dimensional simplicial complex will be said to have $\lambda$-local spectral expansion if:
\begin{itemize}
\item The $1$-skeleton of $X$ and the $1$-skeletons of all its links (in all dimensions $>0$) are connected. 
\item Every $1$-dimensional link of $X$ has a spectral gap $\geq \lambda$, i.e.,
$$\forall \sigma \in X^{(n-2)}, \lambda (X_{\sigma} ) \geq \lambda .$$
\end{itemize} 
For $\lambda >\frac{n-1}{n} ,\kappa <2$, a pure $n$-dimensional simplicial complex will be said to have two sided $(\lambda, \kappa)$-local spectral expansion if:
\begin{itemize}
\item The $1$-skeleton of $X$ and the $1$-skeletons of all its links (in all dimensions $>0$) are connected.  
\item The non zero spectrum of every $1$-dimensional link is contained in the interval $[\lambda , \kappa]$, i.e.,
$$\forall \sigma \in X^{(n-2)}, \lambda (X_{\sigma} ) \geq \lambda, \kappa (X_{\sigma} ) \leq \kappa .$$
\end{itemize} 
\end{definition}

We remark that for $n=1$, both of the above definitions coincide with the usual definitions for graphs when using the convention $X^{(-1)} = \lbrace \emptyset \rbrace$ (by this convention $X_\emptyset = X$).  

A main advantage of the above definition is that the spectrum of the $1$-dimensional links is usually easy to bound or even calculate explicitly in examples. In this paper we shall show that the local spectral expansion of $X$ implies spectral gaps in all the $1$-skeletons of all links of $X$ (in every dimension) and also spectral gaps of the global high-dimensional Laplacians of $X$. 

\begin{remark}
Our results stated below refer to a spectral gaps of weighted graphs, where the links are weighted according to higher dimensional structure of the simplicial complex. Our choice of weights is not identical to the choice of weights in other works, but the results are transferable to other common choices of weight functions - see remark \ref{weight system remark} below.
\end{remark}

We prove the following results:
\begin{theorem}
\label{Intro-theorem1}
Let $X$ be a pure $n$-dimensional simplicial complex. Assume that there are $\kappa \geq \lambda >\frac{n-1}{n}$ such that $X$ is a $(\lambda,\kappa)$-expander, then for every $k$-dimensional simplex $\tau$ of $X$, 
$$f^{n-k-2} (\lambda) \leq \lambda (X_\tau) \text{ and } \kappa (X_\tau) \leq f^{n-k-2} (\kappa),$$
where $\lambda (X_\tau)$, $\kappa (X_\tau)$ are the smallest and largest positive eigenvalues of the $1$-skeleton of $X_\tau$, $f$ is the function $f(x)=2-\frac{1}{x}$ and $f^l$ denotes $f$ composed with itself $l$ times: explicitly,
$$f^l (x) = \dfrac{(l+1)x-l}{l x -(l-1)}.$$
\end{theorem}

In the case of partite complexes, we prove that if $X$ is $(n+1)$-partite then for every $\tau$ of dimension $k$, $\kappa (X_\tau) = \frac{n-k}{n-k-1}$ (this is a generalization of the fact that the Laplacian of every bipartite graph has an eigenvalue $2$). In the partite case, we can bound the spectrum of the links from above if we omit this largest eigenvalue:
\begin{theorem}
\label{Intro-theorem2}
Let $X$ be a pure $n$-dimensional simplicial complex such that all its links (in all dimensions $>0$) are connected. Assume that $X$ is $(n+1)$-partite and that there is $\lambda >\frac{n-1}{n}$ such that 
$$\forall \sigma \in X^{(n-2)}, \lambda \leq \lambda (X_\sigma),$$
then for every $k$-dimensional simplex $\tau$ of $X$, if $\mu$ is an eigenvalue of the Laplacian of the $1$-skeleton of $X_\tau$ and $\mu < \frac{n-k}{n-k-1}$, then
$$f^{n-k-2} (\lambda) \leq \mu \leq 1-(n-k)(1-f^{n-k-2} (\lambda)),$$
where $f$ is the function $f(x)=2-\frac{1}{x}$ and $f^l$ denotes $f$ composed with itself $l$ times.
\end{theorem}

Combining these results with Garland's method (or similar arguments in the partite case), we are able to derive spectral gap of higher Laplacians based on local spectral expansion - see Corollaries \ref{SpecGapCor}, \ref{contraction corollary} and \ref{norm bound - n+1 partite case} below. In a follow up article, we will show how to use these corollaries to deduce mixing results and geometric overlap given local spectral expansion is partite and non-partite simplicial complexes.

The moral of all the theorems and the corollaries mentioned above is that given that the spectra of the $1$-dimensional links in concentrated near $1$ and that the complex has connected links, we can expect good spectral gaps, i.e., the sepctra is concentrated near $1$, in all the Laplacians defined on the complex (both the higher Laplacians and the Laplacians of the links).

\textbf{Structure of this paper.} Section 2 lays out the framework and notations. Section 3 discusses links of simplcial complexes and the concepts of localization and restriction. Section 4 gives the main results about spectral gaps of the links in the non-partite case. Section 5 gives the main results in the partite case. Section 6 gives the results regarding the spectral gaps of higher Laplcians.

\section{Framework}
The framework suggested here owes its existence to the framework suggested in \cite{BS}. Throughout this paper, $X$ is pure $n$-dimensional finite simplicial complex, i.e., every simplex in $X$ is contained in at least one $n$-dimensional simplex. 
\subsection{Weighted simplicial complexes}
Our results in the the introduction were stated with respect to a specific weight function, which was not explicitly described. Our results hold for any balanced weight function (see definition below) and therefore we will work in the general setting of weighted simplicial complex defined below. The weight function used in the introduction will be defined below as the homogeneous weight function. 

For $-1\leq k\leq n$, denote:
\begin{itemize}
\item $X^{(k)}$ is the set of all $k$-simplices in $X$.
\item $\Sigma(k)$ the set of ordered $k$-simplices, i.e., $\sigma \in \Sigma(k)$ is an ordered $(k+1)$-tuple of vertices that form a $k$-simplex in $X$.
\end{itemize} 
Note the  $\Sigma (-1) = X^{(-1)}$ is just the singleton $\lbrace \emptyset \rbrace$. 
\begin{definition}
\label{weightedDef}
A weight function on a simplicial complex $X$ is a strictly positive function $m : \bigcup_{-1 \leq k \leq n} X^{(k)} \rightarrow \mathbb{R}^+$.
The function $m$ will be called a balanced weight function if for every $-1 \leq k \leq n-1$ and every $\tau \in X^{(k)}$, we have the following equality
$$ \sum_{\sigma \in X^{(k+1)}, \tau \subset \sigma} m( \sigma ) =  m (\tau ) ,$$
where $\tau \subset \sigma$ means that $\tau$ is a face of $\sigma$. A simplicial complex $X$ with a balanced weight function $m$ will be called a weighted simplicial complex. 
\end{definition}

\begin{remark}
The idea of introducing a weight function in order to define the inner-product and combinatorial Laplacians appears already in the work of Ballmann and Swiatkowski \cite{BS}, where a the homogeneous weight function (see definition below) was used. The generalization that we use was first published by Horak and Jost in \cite{HJost}.
\end{remark}

Given a weight function $m$ we can define it on ordered simplices (denoting it again as $m$) as
$$m( (v_0,...,v_k)) = m( \lbrace v_0,...,v_k \rbrace), \forall  (v_0,...,v_k) \in \bigcup_{-1 \leq k \leq n} \Sigma (k).$$
If $m$ is balanced, we have the following equality:
$$\forall  \tau \in \bigcup_{-1 \leq k \leq n-1} \Sigma (k), \sum_{\sigma \in \Sigma (k+1), \tau \subset \sigma} m( \sigma ) = (k+2)!  m (\tau ) ,$$
where $\tau \subset \sigma$ means that all the vertices of $\tau$ are contained in $\sigma$ (with no regard to the ordering). We note that under this equality one can start with a strictly positive function $m: \bigcup_{-1 \leq k \leq n} \Sigma (k)  \rightarrow \mathbb{R}^+$ and get a balanced weight function $m: \bigcup_{-1 \leq k \leq n} X^{(k)} \rightarrow \mathbb{R}^+$:

\begin{proposition}
\label{weightedOrderDef}
Let $m: \bigcup_{-1 \leq k \leq n} \Sigma (k)  \rightarrow \mathbb{R}^+$ be a strictly positive function such that:
\begin{enumerate}
\item For every $1 \leq k \leq n$, and every permutation $\pi \in Sym (\lbrace 0,..,k \rbrace)$ we have 
$$m( (v_0,...,v_k) ) = m( (v_{\pi (0)},...,v_{\pi (k)} )), \forall (v_0,...,v_k) \in \Sigma (k).$$
\item 
$$\forall  \tau \in \bigcup_{-1 \leq k \leq n-1} \Sigma (k), \sum_{\sigma \in \Sigma (k+1), \tau \subset \sigma} m( \sigma ) = (k+2)!  m (\tau ) .$$
\end{enumerate}
Then $m : \bigcup_{-1 \leq k \leq n} X^{(k)} \rightarrow \mathbb{R}^+$ defined as 
$$ m( \lbrace v_0,...,v_k \rbrace) = m( (v_0,...,v_k)), \forall   \lbrace v_0,...,v_k \rbrace \in \bigcup_{-1 \leq k \leq n} X^{(k)},$$
is a balanced weight function.
\end{proposition}

\begin{proof}
Follows from the definition.
\end{proof}

\begin{remark}
Below $m$ will always be considered balanced. Also, as defined above, every time we will say a simplicial complex is weighted, we will mean it has a balanced weight. The reason we introduce the general notion of weight is to compare to other works in which the weight function which is considered is not balanced (see also Remark \ref{weight system remark} below).
\end{remark}

\begin{remark}
From the definition of the balanced weight function $m$, it should be clear that every map  $m : X^{(n)} \rightarrow \mathbb{R}^+$ can be extended in a unique way to a balanced weight function $m: \bigcup_{-1 \leq k \leq n} X^{(k)} \rightarrow \mathbb{R}^+$. 
\end{remark}

\begin{definition}
$m$ is called the homogeneous weight on $X$ if for every $\sigma \in X^{(n)}$, we have $m( \sigma )=1$.
\end{definition}

\begin{proposition}
\label{weight in n dim simplices}
For every $-1 \leq k \leq n$ and every $\tau \in X^{(k)}$ we have that
$$\dfrac{1}{(n-k)!}  m (\tau) =\sum_{\sigma \in X^{(n)}, \tau \subseteq \sigma} m (\sigma ),$$
where $\tau \subseteq \sigma$ means that $\tau$ is a face of $\sigma$. 
\end{proposition}

\begin{proof}
The proof is by induction. For $k=n$ this is obvious. Assume the equality is true for $k+1$, then for $\tau \in X^{(k)}$ we have
\begin{dmath*}
m( \tau ) = \sum_{\sigma \in X^{(k+1)}, \tau \subset \sigma} m(\sigma ) = \sum_{\sigma \in X^{(k+1)}, \tau \subset \sigma} (n-k-1)! \sum_{\eta \in X^{(n)}, \sigma \subset \eta} m(\eta) = (n-k) (n-k-1)! \sum_{\eta \in X^{(n)}, \tau \subset \eta} m(\eta)=   (n-k)! \sum_{\eta \in X^{(n)}, \tau \subset \eta} m(\eta) .
\end{dmath*}
\end{proof}

\begin{corollary}
\label{weight in l dim simplices}
For every $-1 \leq k < l \leq n$ and every $\tau \in X^{(k)}$ we have
$$\dfrac{1}{(l-k)!} m(\tau) = \sum_{\sigma \in X^{(l)}, \tau \subset \sigma} m(\sigma) .$$
\end{corollary}

\begin{proof}
For every $\sigma \in X^{(l)}$ we have
$$\dfrac{1}{(n-l)!} m(\sigma ) = \sum_{\eta \in X^{(n)}, \sigma \subseteq \eta} m(\eta) .$$
Therefore
\begin{dmath*}
\sum_{\sigma \in X^{(l)}, \tau \subset \sigma} m(\sigma) = \sum_{\sigma \in X^{(l)}, \tau \subset \sigma} (n-l)! \sum_{\eta \in X^{(n)}, \sigma \subseteq \eta} m(\eta) = \dfrac{(n-k)!}{(l-k)! (n-k - (l-k) )! } (n-l)! \sum_{\eta \in X^{(n)}, \tau \subseteq \eta} m(\eta) = \dfrac{(n-k)!}{(l-k)!}  \sum_{\eta \in X^{(n)}, \tau \subseteq \eta} m(\eta) = \dfrac{1}{(l-k)!} m (\tau ) .
\end{dmath*}
\end{proof}

From now on, we shall always assume that $X$ is weighted (and that the weight function is balanced). 

\begin{remark}
\label{weight system remark}
The reader should note that in other papers a different variants of the homogeneous weight are used: for instance, in our setting, the homogeneous weight is defined for $\tau \in X^{(k)}$ as 
$$m(\tau) =   (n-k)! \vert \lbrace \sigma \in X^{(n)}: \tau \subseteq \sigma \rbrace \vert.$$
In contrast, in \cite{LMM} the following weight function $w$ is used: for $\tau \in X^{(k)}$, 
$$w(\tau) = \dfrac{\vert \lbrace \sigma \in X^{(n)}: \tau \subseteq \sigma \rbrace \vert}{{n+1 \choose k+1} \vert X^{(n)} \vert}.$$
We note that $w$ is unbalanced, but has the nice property of being a probability measure of $X^{(k)}$ for every $k$. We chose to work only with balanced weights, because our computations below are easier under this assumption (for instance, the weighted Laplacian in every link has eigenvalues between $0$ and $2$). The reader should note that all our results regarding spectral gaps are transferable to the weight systems similar to that of \cite{LMM}, because the homogeneous weight differ from this other variants by a multiplicative constant (that depends on the dimension of the simplex). For instance, $\frac{m(\tau)}{w(\tau)} = \frac{(n+1)!}{(k+1)!} \vert X^{(n)} \vert$ for every $k$-dimensional simplex $\tau$.
\end{remark}

\subsection{Cochains and Laplacians in real coefficients}

For $-1 \leq k\leq n$, denote 
$$C^{k}(X, \mathbb{R}) = \lbrace \phi : \Sigma (k) \rightarrow \mathbb{R} : \phi \text{ is antisymmetric} \rbrace.$$
We recall that $\phi : \Sigma (k) \rightarrow \mathbb{R}$ is called antisymmetric, if for every $(v_0,...,v_k) \in  \Sigma (k)$ and every permutation $\pi \in Sym (\lbrace 0,...,k \rbrace)$, we have 
$$\phi ((v_{\pi (0)},...,v_{\pi (k)})) = sgn (\pi) \phi ((v_0,...,v_k)).$$
Every $\phi \in C^k (X,\mathbb{R})$ is called a \textit{$k$-form}, and $C^k (X, \mathbb{R})$ is called the \textit{space of $k$-forms}. 

For $-1 \leq k \leq n$ define an inner product on $C^{k}(X,\mathbb{R})$ as:
$$\forall \phi, \psi \in C^{k}(X,\mathbb{R}), \left\langle \phi , \psi \right\rangle = \sum_{\tau \in \Sigma (k)} \dfrac{m(\tau)}{(k+1)!} \phi (\tau ) \psi (\tau ).$$
Note that with this inner product $C^{k}(X,\mathbb{R})$ is a (finite dimensional) Hilbert space. Denote the norm induced by this inner product as $\Vert . \Vert$. 
For $-1 \leq k \leq n-1$ define the differential $d_k : C^k (X,\mathbb{R}) \rightarrow C^{k+1} (X,\mathbb{R})$ in the usual way, i.e., for every $\phi \in C^k (X,\mathbb{R})$ and every $(v_0,...,v_{k+1})$,
$$(d_k \phi ) ((v_0,...,v_{k+1} )) = \sum_{i=0}^{k+1} (-1)^i \phi ((v_0,..., \widehat{v_i},...,v_{k+1})).$$

One can easily check that for every $0 \leq k \leq n-1$ we have that $d_{k+1} d_k = 0$ and therefore we can define the cohomology in the usual way:
$$H^k (X, \mathbb{R} ) = \dfrac{\Ker (d_k)}{\im (d_{k-1})} .$$

Next, we describe the discrete Hodge theory in our setting. Define $\delta_k : C^{k+1} (X, \mathbb{R} ) \rightarrow C^{k} (X,\mathbb{R})$ as the adjoint operator of $d_{k}$ (with respect to the inner product we defined earlier on $C^{k} (X, \mathbb{R} ), C^{k-1} (X,\mathbb{R})$). Define further operators $\Delta_k^+, \Delta_k^-, \Delta_k :  C^{k} (X, \mathbb{R} ) \rightarrow C^{k} (X, \mathbb{R} )$ as 
$$\Delta_k^+ = \delta_{k} d_k, \Delta_k^- = d_{k-1} \delta_{k-1}, \Delta_k = \Delta_k^+ + \Delta_k^- .$$
The operators $\Delta_k^+, \Delta_k^-, \Delta_k$ are called the upper Laplacian, the lower Laplacian and the full Laplacian. The reader should note that by definition, all these operators are positive (i.e., self-adjoint with a non-negative spectrum). 
\begin{proposition}
\label{HodgeConsid}
For every $1 \leq k \leq n-1$, $H^k (X, \mathbb{R} )$ is isometrically isomorphic as a Hilbert space to $\Ker (\Delta_k )$ and
$$Spec (\Delta^+_{k-1}) \setminus \lbrace 0 \rbrace \subseteq [a,b] \Leftrightarrow Spec (\Delta^-_k) \setminus \lbrace 0 \rbrace \subseteq [a,b],$$
where $Spec (\Delta^+_{k-1}), Spec (\Delta^-_k) $ are the spectrum of $\Delta^+_{k-1}, \Delta^-_k$. 
\end{proposition}

\begin{proof}
Notice that since $d_k^* = \delta_k$ we have the following:
$$\im (\Delta_k^+)=  (\Ker ( \Delta_k^+))^\perp =(\Ker (d_k))^\perp = \im (\delta_k),$$
$$\im (\Delta_k^-)=  (\Ker ( \Delta_k^-))^\perp =(\Ker (\delta_{k-1}))^\perp = \im (d_{k-1}).$$
Therefore, we have an orthogonal decomposition
$$\Ker (d_k) = \Ker ( \Delta_k^+) = \left( \Ker ( \Delta_k^+) \cap \Ker ( \Delta_k^-) \right) \oplus \im (\Delta_k^-) =\Ker (\Delta_k ) \oplus \im (d_{k-1}).$$
Which yields that the orthogonal projection of $\Ker (d_k)$ on $\Ker (\Delta_k )$ gives rise to an isometric isomorphism between $H^k (X, \mathbb{R} )$ and $\Ker (\Delta_k ) $. $\Ker (\Delta_k )$ is called the space of harmonic $k$-forms on $X$. Also notice that due to the fact that $\Delta^+_{k-1} = \delta_{k-1} d_{k-1}, \Delta^-_{k} =  d_{k-1} \delta_{k-1}$, we have
$$Spec (\Delta^+_{k-1}) \setminus \lbrace 0 \rbrace \subseteq [a,b] \Leftrightarrow Spec (\Delta^-_k) \setminus \lbrace 0 \rbrace \subseteq [a,b],$$
\end{proof}
The next proposition gives an explicit formula for $\delta_k, \Delta_k^+, \Delta_k^-$:

\begin{proposition}

\begin{enumerate}
\item Let $-1 \leq k \leq n-1$ then: $\delta_k :C^{k+1}(X,\mathbb{R})\rightarrow C^{k}(X,\mathbb{R})$ is \[
\delta_k \phi(\tau)=\sum_{\begin{array}{c}
{\scriptstyle v \in\Sigma(0)}\\
{\scriptstyle v \tau\in \Sigma (k+1)}\end{array}}\frac{m(v\tau)}{m(\tau)}\phi(v\tau),\;\tau\in\Sigma(k)\]
where $v\tau=(v,v_{0},...,v_{k})$ for $\tau=(v_{0},...,v_{k})$. 
\item For $0 \leq k \leq n-1$, $\phi\in C^{k}(X,\mathbb{R})$ and $\sigma\in\Sigma(k)$,
\[
\Delta_k^+ \phi (\sigma) = \phi(\sigma)-\sum_{\begin{array}{c}
{\scriptstyle v \in\Sigma(0)}\\
{\scriptstyle v \sigma \in \Sigma (k+1)}\end{array} }\sum_{0\leq i\leq k}(-1)^{i}\frac{m(v\sigma)}{m(\sigma)}\phi(v\sigma_{i}) ,\]
where $\sigma_i = (v_0,..., \widehat{v_i},...,v_{k})$ for $\sigma =(v_{0},...,v_{k})$.
\item For  $0 \leq k \leq n$, $\phi\in C^{k}(X,\mathbb{R})$ and $\sigma\in\Sigma(k)$,
$$ \Delta_k^- \phi (\sigma ) = \sum_{i=0}^k (-1)^i  \sum_{ v \in\Sigma(0),  v\sigma_i \in\Sigma(k) }\frac{m(v\sigma_i)}{m(\sigma_i)}\phi(v\sigma_i) ,$$
where $\sigma_i = (v_0,..., \widehat{v_i},...,v_{k})$ for $\sigma =(v_{0},...,v_{k})$. 
\end{enumerate}
\end{proposition}

\begin{proof}
\begin{enumerate}

\item For $\sigma \in \Sigma (k+1)$ and $\tau \subset \sigma, \tau \in \Sigma (k)$ denote by $[\sigma : \tau ]$ the incidence coefficient of $\tau$ with respect to $\sigma$, i.e., if for $\sigma = (v_{0},...,v_{k+1})$, $\tau =  \lbrace v_0,..., \widehat{v_i},...,v_{k+1} \rbrace$, then for every $\psi \in C^k (X, \mathbb{R})$, we have that $[ \sigma : \tau ] \psi (\tau ) = (-1)^i \psi (\sigma_i)$. Take $\phi \in C^{k+1} (X, \mathbb{R})$ and $\psi \in C^k (X, \mathbb{R})$, then we have
\begin{dmath*}
\left\langle d \psi , \phi \right\rangle = \sum_{\sigma \in \Sigma (k+1)} \dfrac{m(\sigma )}{(k+2)!} \left( \sum_{i=0}^{k+1}(-1)^{i}\psi(\sigma_{i}) \right) \phi ( \sigma )  = 
{ \sum_{\sigma \in \Sigma (k+1)} \dfrac{m(\sigma )}{(k+1)! (k+2)! } \left( \sum_{ \tau \in\Sigma(k), \tau \subset \sigma} [\sigma : \tau] \psi(\tau) \right)  \phi ( \sigma )}  = 
 \sum_{\sigma \in \Sigma (k+1)} \dfrac{m(\tau )}{(k+1)!} \sum_{ \tau \in\Sigma(k) , \tau \subset \sigma}    \psi(\tau) \left( \dfrac{[\sigma : \tau] m(\sigma )}{m(\tau ) (k+2)!} \phi ( \sigma ) \right) = 
\sum_{\tau \in \Sigma (k)} \dfrac{m(\tau )}{(k+1)!} \sum_{ \sigma \in\Sigma(k+1), \tau \subset \sigma }   \psi(\tau)\left( \dfrac{[\sigma : \tau] m(\sigma )}{m(\tau ) (k+2)!} \phi ( \sigma ) \right) = 
\sum_{\tau \in \Sigma (k)} \dfrac{m(\tau )}{(k+1)! }    \psi(\tau) \left( \sum_{ \sigma \in\Sigma(k+1), \tau \subset \sigma} \dfrac{[\sigma : \tau] m(\sigma )}{m(\tau ) (k+2)!} \phi ( \sigma ) \right) = 
\sum_{\tau \in \Sigma (k)} \dfrac{m(\tau )}{(k+1)! }   \psi(\tau) \left( \sum_{ v \in\Sigma(0),  v\tau\in\Sigma(k+1) }\frac{m(v\tau)}{m(\tau)}\phi(v\tau) \right) =
 \left\langle \psi, \sum_{ v \in\Sigma(0),  v\tau\in\Sigma(k+1) }\frac{m(v\tau)}{m(\tau)}\phi(v\tau) \right\rangle .
 \end{dmath*}

\item For every $\phi \in C^k (X, \mathbb{R})$ and every $\sigma \in \Sigma (k)$ we have:
\begin{dmath*}
\delta d \phi (\sigma ) = \sum_{\begin{array}{c}
{\scriptstyle v \in\Sigma(0)}\\
{\scriptstyle v\sigma \in\Sigma(k+1) }\end{array} }\frac{m(v\sigma)}{m(\sigma)} d \phi(v\sigma) = 
\sum_{\begin{array}{c}
{\scriptstyle v \in\Sigma(0)}\\
{\scriptstyle v\sigma \in\Sigma(k+1) }\end{array}}\frac{m(v\sigma)}{m(\sigma)} \phi (\sigma ) - \sum_{\begin{array}{c}
{\scriptstyle v \in\Sigma(0)}\\
{\scriptstyle v\sigma \in\Sigma(k+1) }\end{array} }\sum_{0\leq i\leq k}(-1)^{i}\frac{m(v\sigma)}{m(\sigma)}\phi(v\sigma_{i}) =
\sum_{\begin{array}{c}
{\scriptstyle \gamma \in\Sigma(k+1)}\\
{\scriptstyle \sigma \subset \gamma }\end{array}}\frac{m(\gamma)}{(k+2)! m(\sigma)} \phi (\sigma ) - \sum_{\begin{array}{c}
{\scriptstyle v \in\Sigma(0)}\\
{\scriptstyle v\sigma \in\Sigma(k+1) }\end{array}}\sum_{0\leq i\leq k}(-1)^{i}\frac{m(v\sigma)}{m(\sigma)}\phi(v\sigma_{i}) =
\phi (\sigma ) - \sum_{\begin{array}{c}
{\scriptstyle v \in\Sigma(0)}\\
{\scriptstyle v\sigma \in\Sigma(k+1) }\end{array}}\sum_{0\leq i\leq k}(-1)^{i}\frac{m(v\sigma)}{m(\sigma)}\phi(v\sigma_{i}).
\end{dmath*}
\item For every $\phi \in C^k (X, \mathbb{R})$ and every $\sigma \in \Sigma (k)$ we have: 
$$d \delta \phi (\sigma) = \sum_{i=0}^k (-1)^i \delta \phi (\sigma_i) = \sum_{i=0}^k (-1)^i  \sum_{ v \in\Sigma(0),  v\sigma_i \in\Sigma(k) }\frac{m(v\sigma_i)}{m(\sigma_i)}\phi(v\sigma_i) .$$

\end{enumerate}
\end{proof}

Note that by the above proposition, we have for every $\phi \in  C^{0} (X, \mathbb{R} )$ that
$$\delta_{-1} \phi (\emptyset ) = \sum_{v \in \Sigma (0)} \dfrac{m(v)}{m (\emptyset )} \phi (v),$$
and
$$\forall u \in \Sigma (0), \Delta_0^- \phi (u) = \delta_0 \phi (\emptyset).$$ 

\begin{proposition}
\label{Delta0Norm}
For every $\phi \in  C^{0} (X, \mathbb{R} )$, $\left\langle \Delta_0^- \phi , \phi \right\rangle = \Vert \delta_{-1} \phi \Vert^2 =  \Vert \Delta_0^- \phi \Vert^2$.
\end{proposition}

\begin{proof}
For every $\phi \in  C^{0} (X, \mathbb{R} )$ we have
\begin{dmath*}
\left\langle \Delta_0^- \phi , \phi \right\rangle = \sum_{u \in \Sigma (0)} m(u) \left( \sum_{v \in \Sigma (0)} \dfrac{m(v)}{m (\emptyset )} \phi (v) \right) \phi (u)  =  \left( \sum_{v \in \Sigma (0)} \dfrac{m(v)}{m (\emptyset )} \phi (v) \right) \sum_{u \in \Sigma (0)} m(u) \phi (u) = m( \emptyset ) \left( \sum_{v \in \Sigma (0)} \dfrac{m(v)}{m (\emptyset )} \phi (v) \right)^2 = \Vert \delta_{-1} \phi \Vert^2.
\end{dmath*}
Also note that 
$$ \Vert \delta_{-1} \phi \Vert^2 = m( \emptyset ) \left( \sum_{v \in \Sigma (0)} \dfrac{m(v)}{m (\emptyset )} \phi (v) \right)^2  = \sum_{u \in \Sigma (0)} m(u)  \left( \sum_{v \in \Sigma (0)} \dfrac{m(v)}{m (\emptyset )} \phi (v) \right)^2 = \Vert \Delta_0^- \phi \Vert^2.$$
\end{proof}

\begin{proposition}
\label{ProjOnConstProp}
For every $\phi \in  C^{0} (X, \mathbb{R} )$, $\Delta_0^- \phi$ is the projection of $\phi$ on the space of constant $0$-forms.
\end{proposition}

\begin{proof}
Let $\textbf{1} \in C^0 (X, \mathbb{R})$ be the constant function $\textbf{1} (u) = 0, \forall u \in \Sigma (0)$. Then the projection of $\phi$ on the space of constant $0$-forms is given by $\frac{\left\langle \phi , \textbf{1} \right\rangle}{\Vert \textbf{1} \Vert^2} \textbf{1}$. Note that
$$ \Vert \textbf{1} \Vert^2 = \sum_{v \in \Sigma (0)} m(v) = m (\emptyset ),$$
$$ \left\langle \phi , \textbf{1} \right\rangle = \sum_{v \in \Sigma (0)} m(v) \phi (v).$$
Therefore for every $u \in \Sigma (0)$,
$$ \dfrac{\left\langle \phi , \textbf{1} \right\rangle}{\Vert \textbf{1} \Vert^2} \textbf{1} (u) = \sum_{v \in \Sigma (0)} \dfrac{m(v)}{m (\emptyset )} \phi (v) = \Delta_0^- (u).$$
\end{proof}

\begin{remark}
We remark that for $\Delta_0^+$ one always have $\Vert \Delta_0^+ \Vert \leq 2$, where $\Vert . \Vert$ here denotes the operator norm (we leave this calculation to the reader). We also remark that the largest eigenvalue of $\Delta_0^+$ is always larger than $1$. This can be seen easily when thinking about  $\Delta_0^+$ in matrix form: it is a matrix with $1$ along the diagonal and $0$ as an eigenvalue. Since the trace of $\Delta_0^+$ as a matrix is equal to the sum of eigenvalues, we can deduce it must have at least one eigenvalue strictly larger than $1$.  
\end{remark}

From now on, when there is no chance of confusion, we will omit the index of $d_k, \delta_k, \Delta^+_k, \Delta^-_k, \Delta_k $ and just refer to them as $d, \delta, \Delta^+, \Delta^-, \Delta $.

\subsection{Partite simplicial complexes}

In important source of examples of simplicial complexes which act like bipartite expander graphs comes from $(n+1)$-partite $n$-dimensional simplicial complexes:
\begin{definition}
An $n$-dimensional simplicial complex $X$ will be called $(n+1)$-partite, if there is a disjoint partition $X^{(0)} = S_0 \cup ... \cup  S_n$ such that for every $u, v \in X^{(0)}$, 
$$\lbrace u ,v \rbrace \in X^{(1)} \Rightarrow \exists \, 0 \leq i, j \leq n, i \neq j, u \in S_i, v \in S_j.$$
If $X$ is pure $n$-dimensional, the above condition is equivalent to the following condition:
$$ \lbrace u_0,...,u_n \rbrace \in X^{(n)} \Rightarrow \exists \, \pi \in Sym (\lbrace 0,...,n \rbrace ), \forall 0 \leq i \leq n, u_i \in  S_{\pi (i)} .$$
We shall call $S_0,...,S_n$ the sides of $X$.
\end{definition}

Let $X$ be a pure $n$-dimensional, weighted, $(n+1)$-partite simplicial complex with sides $S_0,...,   S_n$ as in the above definition. We shall define the following operators: \\
For $0 \leq j \leq n$ and $-1 \leq k \leq n-1$, define
$$d_{(k,j)} : C^k (X, \mathbb{R}) \rightarrow   C^{k+1} (X, \mathbb{R}) ,$$
as following:
$$d_{(k,j)} \phi ((v_0,...,v_{k+1})) = \begin{cases}
0 &  v_0 \notin S_j,...,v_{k+1} \notin S_j \\
(-1)^i \phi ((v_0,...,\widehat{v_i},...,v_{k+1})) & v_i \in S_j
\end{cases} .$$
Denote by $\delta_{(k,j)} :  C^{k+1} (X, \mathbb{R}) \rightarrow   C^{k} (X, \mathbb{R})$ the adjoint operator to $d_{(k,j)}$ and $\Delta^{-}_{(k,j)} = d_{(k-1,j)} \delta_{(k-1,j)}$.

\begin{proposition}
Let $-1 \leq k \leq n, 0 \leq j \leq n$, then for every $\phi \in C^{k+1} (X, \mathbb{R})$ 
$$\delta_{(k,j)} \phi (\tau) = \sum_{v \in S_j, v \tau \in \Sigma (k+1)} \dfrac{m(v \tau)}{m(\tau)} \phi (v \tau).$$
\end{proposition}

\begin{proof}
Let $\phi \in C^{k+1} (X, \mathbb{R}) , \psi \in C^{k} $, then 
\begin{dmath*}
\left\langle d_{(k,j)} \psi, \phi \right\rangle = \sum_{\sigma \in \Sigma (k+1)} \dfrac{m(\sigma )}{(k+2)!} d_{(k,j)} \psi (\sigma) \phi (\sigma) = \sum_{\sigma = (v_0,...,v_{k+1}) \in \Sigma (k+1), v_i \in S_j} \dfrac{m(\sigma )}{(k+2)!} (-1)^i \psi (\sigma_i) \phi (\sigma) = \sum_{\sigma = (v_0,...,v_{k+1}) \in \Sigma (k+1), v_i \in S_j} \dfrac{m(v_i \sigma_i )}{(k+2)!}  \psi (\sigma_i) \phi (v_i \sigma_i) =
\sum_{\tau \in \Sigma (k)} \sum_{v \in S_j} \dfrac{m(v \tau )}{(k+1)!}  \psi (\tau) \phi (v \tau) = 
\sum_{\tau \in \Sigma (k)} \dfrac{m( \tau )}{(k+1)!}  \psi (\tau) \left( \sum_{v \in S_j}  \dfrac{m(v \tau )}{m (\tau)} \phi (v \tau) \right) .
\end{dmath*}
\end{proof}

A straightforward computation gives rise to:
\begin{corollary}
\label{side average operators}
For every $0 \leq k \leq n$, $0 \leq j \leq n$ and every $\phi \in C^{k} (X,\mathbb{R})$ we have that
$$\Delta^{-}_{(k,j)} \phi (\sigma) =
\begin{cases}
0 & \sigma = (v_0,...,v_k), \forall i, v_i \notin S_j \\
(-1)^i \sum_{v \in S_j, v \sigma_i \in \Sigma (k)} \dfrac{m(v \sigma_i)}{m(\sigma_i)} \phi (v \sigma_i) & \sigma = (v_0,...,v_k), v_i \in S_j 
\end{cases}.$$
\end{corollary}


\section{Links of $X$}

Let $\lbrace v_{0},...,v_{j} \rbrace=\tau\in X^{(j)}$, denote by $X_{\tau}$
the \emph{link} of $\tau$ in $X$, that is, the (pure) complex of dimension
$n-j-1$ consisting on simplices $\sigma=\lbrace  w_{0},...,w_{k} \rbrace$
such that $\lbrace v_{0},...,v_{j} \rbrace, \lbrace w_{0},...,w_{k} \rbrace$ are disjoint as sets and $\lbrace v_{0},...,v_{j} \rbrace \cup \lbrace w_{0},...,w_{k} \rbrace \in X^{(j+k+1)}$. Note that for $\emptyset \in \Sigma (-1)$, $X_\emptyset =X$. \\
For an ordered simplex $( v_{0},...,v_{j} ) \in \Sigma (k)$ define $X_{( v_{0},...,v_{j} ) } = X_{\lbrace v_{0},...,v_{j} \rbrace }$. \\
Throughout this article we shall assume that all the $1$-skeletons of the links of $X$ of dimension $>0$ are connected . \\
Next, we will basically repeat the definitions that we gave above for $X$: \\  
For $0\leq k\leq n-j-1$, denote by $\Sigma_{\tau}(k)$ the set of
ordered $k$-simplices in $X_\tau$. \\
Define the function $m_\tau : \bigcup_{0 \leq k \leq n-j-1} \Sigma_\tau (k) \rightarrow \mathbb{R}^+$ as
$$\forall \sigma \in \Sigma_\tau (k), m_{\tau}(\sigma) = m (\tau \sigma ),$$
where $\tau \sigma$ is the concatenation of $\tau$ and $\sigma$, i.e., if $\tau = (v_{0},...,v_{j} ), \sigma = (w_{0},...,w_{k} )$ then  $ \tau \sigma = (v_{0},...,v_{j},w_{0},...,w_{k} )$.
\begin{proposition}
The function $m_\tau$ defined above is a balanced weight function of $X_\tau$.
\end{proposition}

\begin{proof}
The fact that $m_\tau$ is invariant under permutation is obvious, therefore we are left to check that for every $\eta \in \Sigma_\tau (k)$ we have
$$ \sum_{\sigma \in \Sigma_\tau (k+1), \eta \subset \sigma} m_\tau ( \sigma ) = (k+2)!  m_\tau (\eta ).$$
For $\eta \in \Sigma_\tau (k)$ we have by definition
\begin{dmath*}
 \sum_{\sigma \in \Sigma_\tau (k+1), \eta \subset \sigma } m_\tau (\sigma)= \sum_{ \sigma \in \Sigma_\tau (k+1), \eta \subset \sigma } m (\tau \sigma) = \sum_{ \gamma \in \Sigma (j+k+2), \tau \eta \subset \gamma } \dfrac{ (k+2)!}{(j+k+3)!} m ( \gamma ) =
= { (k+2)! m (\tau \eta ) =  (k+2)!  m_\tau ( \eta ).}
\end{dmath*}

\end{proof}

We showed that $X_\tau$ is a weighted simplicial complex with the weight function $m_\tau$ and therefore we can repeat all the definitions given before for $X$. Therefore we have:
\begin{enumerate}
\item $C^{k}(X_{\tau},\mathbb{R})$ with the inner product on it.
\item Differential $d_{\tau,k} : C^{k}(X_{\tau},\mathbb{R}) \rightarrow C^{k+1}(X_{\tau},\mathbb{R})$, $\delta_{\tau,k} = (d_{\tau,k})^*. \delta_{\tau,0}$.
\item $\Delta_{\tau,k}^+ = \delta_{\tau,k} d_{\tau,k},  \Delta_{\tau,k}^- = d_{\tau ,k-1} \delta_{\tau,k-1}, \Delta_{\tau,k} = \Delta_{\tau,k}^+ + \Delta_{\tau,k}^-$.
\end{enumerate} 

From now on, when there is no chance of confusion, we will omit the index of $d_{\tau,k}, \delta_{\tau,k}, \Delta^+_{\tau,k}, \Delta^-_{\tau,k}, \Delta_{\tau,k} $ and just refer to them as $d_\tau, \delta_\tau, \Delta^+_\tau, \Delta^-_\tau, \Delta_\tau $.
\begin{remark}
Notice that for an $n$-dimensional simplicial complex $X$, if $m$ is homogeneous, then for every $\tau \in \Sigma (n-2)$, $X_\tau$ is a graph such that $m_\tau$ assigns the value $1$ for every edge. In this case, $\Delta_{\tau,0}^+$ is the usual graph Laplacian.
\end{remark}
We now turn to describe how maps $C^{k}(X,\mathbb{R})$ induce maps on the links of $X$. This is done in two different ways described below: localization and restriction. 

\subsection{Localization and Garland's method}
\begin{definition}
For $\tau \in \Sigma (j)$ and $j-1 \leq k \leq n$ define the \emph{localization map} \[
C^{k}(X,\mathbb{R})\rightarrow C^{k-j-1}(X_{\tau},\mathbb{R}),\;\phi\rightarrow\phi_{\tau} ,\]
where $\phi_\tau$ is defined by $\phi_{\tau}(\sigma)=\phi(\tau\sigma)$. 
\end{definition}
When $\phi \in C^k (X, \mathbb{R} ), k>0$, one can compute $\Vert \phi \Vert^2, \Vert \delta \phi \Vert^2, \Vert d \phi \Vert^2$ by using all the localizations of the form $\phi_\tau, \tau \in \Sigma (k-1)$. After these calculations, one can also bound the spectrum of high Laplacians $\Delta^+_k$, based on the spectral gap in the links. This idea was introduced by Garland in \cite{Gar} and is known today as Garland's method. This is a well known method (see for instance \cite{BS} and \cite{Borel}) and therefore we state the most of the results below without proofs.

\begin{lemma}
\label{LocalizNorm1}
For every $0 \leq k \leq n$ and every $\phi, \psi \in C^k (X, \mathbb{R} )$, one has:
\begin{enumerate}
\item $(k+1)! \left\langle \phi , \psi \right\rangle= \sum_{\tau \in \Sigma(k-1)} \left\langle \phi_\tau , \psi_\tau \right\rangle .$
\item For $\tau \in \Sigma (k-1)$, 
$k! \left\langle \delta \phi , \delta \psi \right\rangle= \sum_{\tau \in \Sigma(k-1)} \left\langle \delta_\tau \phi_\tau , \delta_\tau \psi_\tau \right\rangle .$

\end{enumerate}
\end{lemma}

\begin{lemma}
\label{LocalizNorm2}
For every $0 \leq k \leq n-1$ and every $\phi, \psi \in C^k (X, \mathbb{R} )$, one has:
 $$  k! \left\langle d \phi , d \psi \right\rangle =  \sum_{\tau \in \Sigma (k-1)} \left( \left\langle d_\tau \phi_\tau , d_\tau \psi_\tau \right\rangle -  \dfrac{ k}{k+1} \left\langle \phi_\tau , \psi_\tau \right\rangle \right).$$
 \end{lemma}

\begin{corollary}
\label{LocalizNorm3}
For every $1 \leq k \leq n$ and every $\phi, \psi \in C^k (X, \mathbb{R} )$, one has:
 $$  k! \langle d \phi , d \psi \rangle + k!k \langle \phi , \psi \rangle =  \sum_{\tau \in \Sigma (k-1)} \langle d_\tau \phi_\tau, d_\tau \psi_\tau \rangle .$$
In particular, for $\phi = \psi$, one has:
 $$  k! \Vert d \phi \Vert^2 + k!k \Vert \phi \Vert^2 =  \sum_{\tau \in \Sigma (k-1)} \Vert d_\tau \phi_\tau \Vert^2 .$$
\end{corollary} 


\begin{lemma}
\label{SpecGapLocalToGlobal1}
Let $X$ be a pure $n$-dimensional weighted simplicial complex such that all the links of $X$ of dimension $>0$ are connected. Also, assume that $n>1$. For $0 \leq k \leq n-1$, if there are $\kappa \geq \lambda >0$ such that
$$\bigcup_{\tau \in \Sigma (k-1)} Spec (\Delta_{\tau,0}^+) \setminus \lbrace 0 \rbrace \subseteq [\lambda, \kappa],$$
then for every $\phi \in C^k (X, \mathbb{R})$ we have
$$(k+1) \Vert \phi \Vert^2 \left( \kappa - \dfrac{k}{k+1} \right) - \kappa  \Vert \delta \phi \Vert^2  \geq    \Vert d \phi \Vert^2  \geq (k+1) \Vert \phi \Vert^2 \left( \lambda - \dfrac{k}{k+1} \right) - \lambda  \Vert \delta \phi \Vert^2,$$
or equivalently,
\begin{align*}
(k+1) \Vert \phi \Vert^2 \left( \kappa - \dfrac{k}{k+1} \right) - \kappa  \langle \Delta^-_k \phi, \phi \rangle   \geq    \langle \Delta^+_k \phi, \phi \rangle  \geq \\
(k+1) \Vert \phi \Vert^2 \left( \lambda - \dfrac{k}{k+1} \right) - \lambda  \langle \Delta^-_k \phi, \phi \rangle.
\end{align*}
\end{lemma}

\begin{proof}

Let $0 \leq k \leq n-1$. Fix some $\tau \in \Sigma (k-1)$ and some $\phi \in C^k (X, \mathbb{R})$. For $\phi_\tau$ recall that, by Proposition \ref{ProjOnConstProp}, $\Delta_{\tau,0}^- \phi_\tau$ is the projection of $\phi_\tau$ on the space of constant functions. Denote by $(\phi_\tau)^1$ the orthogonal complement of this projection. Since the $1$-skeleton of $X_\tau$ is connected, we have that $\Ker (\Delta_{\tau,0}^+)$ is exactly the space of constant functions and therefore 
$$\kappa \Vert (\phi_\tau)^1 \Vert^2 \geq \left\langle \Delta_{\tau, 0}^+ \phi_\tau , \phi_\tau \right\rangle \geq \lambda \Vert (\phi_\tau)^1 \Vert^2.$$
Note that $\Vert (\phi_\tau)^1 \Vert^2 =\Vert \phi_\tau \Vert^2 - \Vert \Delta_{\tau,0}^- \phi_\tau \Vert^2$ and that $\left\langle \Delta_{\tau, 0}^+ \phi_\tau , \phi_\tau \right\rangle = \Vert d_\tau \phi_\tau \Vert^2$. Therefore
$$\kappa \Vert \phi_\tau \Vert^2 - \Vert \Delta_{\tau,0}^- \phi_\tau \Vert^2  \geq \Vert d_\tau \phi_\tau \Vert^2 \geq \lambda \Vert \phi_\tau \Vert^2 - \Vert \Delta_{\tau,0}^- \phi_\tau \Vert^2.$$
Since the above inequality is true for every $\tau \in \Sigma (k-1)$, we can sum over all $\tau \in \Sigma (k-1)$ and get
$$\kappa \sum_{\tau \in \Sigma (k-1)} \left( \Vert \phi_\tau \Vert^2 - \Vert \Delta_{\tau,0}^- \phi_\tau \Vert^2 \right) \geq \sum_{\tau \in \Sigma (k-1)}  \Vert d_\tau \phi_\tau \Vert^2 \geq \lambda \sum_{\tau \in \Sigma (k-1)}  \left( \Vert \phi_\tau \Vert^2 - \Vert \Delta_{\tau,0}^- \phi_\tau \Vert^2 \right).$$
By Proposition \ref{Delta0Norm}, we have that $\Vert  \Delta_{\tau,0}^- \phi_\tau \Vert^2 = \Vert \delta_{\tau,0} \phi_\tau \Vert^2$, therefore we can write 
$$\kappa \sum_{\tau \in \Sigma (k-1)} \left( \Vert \phi_\tau \Vert^2 - \Vert \delta_{\tau,0} \phi_\tau \Vert^2 \right) \geq \sum_{\tau \in \Sigma (k-1)}  \Vert d_\tau \phi_\tau \Vert^2 \geq \lambda \sum_{\tau \in \Sigma (k-1)}  \left( \Vert \phi_\tau \Vert^2 - \Vert \delta_{\tau,0} \phi_\tau \Vert^2 \right).$$
By Lemma \ref{LocalizNorm1}, applied for $\phi = \psi$, we have that 
$$\sum_{\tau \in \Sigma (k-1)} \left( \Vert \phi_\tau \Vert^2 - \Vert \delta_{\tau,0} \phi_\tau \Vert^2 \right) = (k+1)! \Vert \phi \Vert^2 - k! \Vert \delta \phi \Vert^2.$$
Therefore
$$\kappa \left( (k+1)! \Vert \phi \Vert^2 - k! \Vert \delta \phi \Vert^2 \right) \geq \sum_{\tau \in \Sigma (k-1)}  \Vert d_\tau \phi_\tau \Vert^2 \geq \lambda \left( (k+1)! \Vert \phi \Vert^2 - k! \Vert \delta \phi \Vert^2 \right).$$
By Corollary \ref{LocalizNorm3} we have for every $\phi \in C^k (X, \mathbb{R})$ that
  $$  k! \Vert d \phi \Vert^2 + k!k \Vert \phi \Vert^2 =  \sum_{\tau \in \Sigma (k-1)} \Vert d_\tau \phi_\tau \Vert^2 ,$$ 
and therefore
$$\kappa \left( (k+1)! \Vert \phi \Vert^2 - k! \Vert \delta \phi \Vert^2 \right) \geq   k! \Vert d \phi \Vert^2 + k!k \Vert \phi \Vert^2 \geq \lambda \left( (k+1)! \Vert \phi \Vert^2 - k! \Vert \delta \phi \Vert^2 \right).$$
Dividing by $k!$ and then subtracting $k \Vert \phi \Vert^2$ gives the inequality stated in the lemma.
\end{proof}

\begin{corollary}
\label{SpecGapLocalToGlobal2}
Let $X$ as in the above lemma. For $0 \leq k \leq n-1$, if there are $\kappa \geq \lambda >\frac{k}{k+1}$ such that
$$\bigcup_{\tau \in \Sigma (k-1)} Spec (\Delta_{\tau, 0}^+) \setminus \lbrace 0 \rbrace \subseteq [\lambda, \kappa],$$
then there is an orthogonal decomposition $C^k (X,\mathbb{R}) = \Ker (\Delta_k^+) \oplus \Ker (\Delta_k^-)$  and 
$$ Spec (\Delta_{k}^+) \setminus \lbrace 0 \rbrace \subseteq [(k+1 )\lambda - k, (k+1) \kappa- k], $$
$$Spec (\Delta_{k+1}^-) \setminus \lbrace 0 \rbrace \subseteq [(k+1 )\lambda - k, (k+1) \kappa- k] .$$
\end{corollary}

Lemma \ref{SpecGapLocalToGlobal1} has the following corollary:

\begin{corollary}
\label{norm bound - local to global}
Let $X$ as in the above lemma. For $0 \leq k \leq n-1$, if there are $\kappa \geq \lambda >\frac{k}{k+1}$ such that
$$\bigcup_{\tau \in \Sigma (k-1)} Spec (\Delta_{\tau, 0}^+) \setminus \lbrace 0 \rbrace \subseteq [\lambda, \kappa],$$
then 
$$\left\Vert \Delta_k^+ + \dfrac{\lambda + \kappa}{2} \Delta_k^- - (k+1) (\dfrac{\lambda + \kappa}{2} - \dfrac{k}{k+1}) I \right\Vert \leq (k+1) \dfrac{\kappa - \lambda}{2},$$
where $\Vert. \Vert$ denotes the operator norm.
\end{corollary}

\begin{proof}
By Lemma \ref{SpecGapLocalToGlobal1}, 
$$(k+1) \dfrac{\kappa - \lambda}{2} \geq \langle (\Delta_k^+ + \dfrac{\lambda + \kappa}{2} \Delta_k^- - (k+1) (\dfrac{\lambda + \kappa}{2} - \dfrac{k}{k+1}) I) \phi, \phi \rangle \geq 
-(k+1) \dfrac{\kappa - \lambda}{2}.$$
Note that $\Delta_k^+ + \frac{\lambda + \kappa}{2} \Delta_k^- - (k+1) (\frac{\lambda + \kappa}{2} - \frac{k}{k+1}) I $ is a self adjoint operator and therefore the above inequality yields the needed inequality.
\end{proof}

\subsection{Localization in partite complexes} 
\label{Localization in partite complexes subsection}
In the case of partite complexes applying Corollary \ref{norm bound - local to global} can be rather uninformative, because the largest eigenvalue in the links can be quite large. Consider for instance the case of $1$-dimensional links, which are bipartite graphs and therefore the largest eigenvalue of the graph Laplacian in those links is $2$. This phenomena of large eigenvalue generalize to every link of a partite complex (see Proposition \ref{upper bound of the spectrum in partite complexes - prop} below). The idea below is to derive a result similar to Corollary \ref{norm bound - local to global}, that does not consider the high end of the spectrum in partite complexes.

Let $X$ be a pure $n$-dimensional, weighted, $(n+1)$-partite simplicial complex with sides $S_0,...,S_n$. Notice that for any $-1 \leq k \leq n-1$, $X_\tau$ is a $(n-k)$-partite complex. In order to keep the indexing of the sides consistent, we shall denote as follows: for $\tau = (v_0,...,v_k), v_i \in S_{j_i} $, the sides of $X_\tau$ will be denoted by $S_{\tau,j}$, where $j \neq j_0,...,j_k$ and $S_{\tau, j} \subseteq S_{j}$.

This will allow us to define $d_{\tau, (l,j)}, \delta_{\tau, (l,j)}$ on $X_\tau$ for $-1 \leq l \leq n-k-1$ in the following way: if $\tau = (v_0,...,v_k), v_i \in S_{j_i} $, then for $j \neq  j_0,...,j_k$, define $d_{\tau, (l,j)}, \delta_{\tau, (l,j)}$ as above (using the indexing on $X_\tau$). If $j = j_i$ for some $0 \leq i \leq k$, then define $d_{\tau, (l,j)} \equiv 0, \delta_{\tau, (l,j)} \equiv 0$. Denote $\Delta^-_{\tau, (l,j)} = d_{\tau, (l-1,j)} \delta_{\tau, (l-1,j)}$.  

After setting these conventions, we will show the following: 
\begin{proposition}
\label{partite localization of Delta^-}
Let $X$ be a pure $n$-dimensional, weighted, $(n+1)$-partite simplicial complex. Then for every $\phi \in  C^{k} (X,\mathbb{R} )$ and every $0 \leq j \leq n$, we have that
$$k! \langle \Delta^-_{(k,j)} \phi, \phi \rangle = \sum_{\tau \in \Sigma (k-1)} \langle \Delta^-_{\tau, (0,j)} \phi_\tau, \phi_\tau \rangle .$$ 
\end{proposition} 
 
\begin{proof}
Let $\phi \in  C^{k} (X,\mathbb{R} )$, then by definition
\begin{dmath*}
k! \langle \Delta^-_{(k,j)} \phi, \phi \rangle = k! \langle \delta^-_{(k,j)} \phi, \delta^-_{(k,j)} \phi \rangle = \sum_{\tau \in \Sigma (k-1)} m (\tau) \left( \sum_{v \in S_j, v \tau \in \Sigma (k+1)} \dfrac{m(v \tau)}{m(\tau)} \phi (v \tau) \right)^2 = \sum_{\tau \in \Sigma (k-1)} m_\tau (\emptyset) \left( \sum_{v \in S_j, v  \in \Sigma_\tau (0)} \dfrac{m_{\tau} (v)}{m_{\tau} (\emptyset)} \phi_\tau (v) \right)^2 = \sum_{\tau \in \Sigma (k-1)} \Vert \delta_{\tau, (0,j)} \phi_\tau \Vert^2 = \sum_{\tau \in \Sigma (k-1)} \langle \Delta^-_{\tau, (0,j)} \phi_\tau, \phi_\tau \rangle .
\end{dmath*}
\end{proof} 

The next theorem is the $(n+1)$-partite analogue of Corollary \ref{norm bound - local to global}:
\begin{theorem}
\label{Contraction in partite case thm}
Let $X$  be a pure $n$-dimensional, $(n+1)$-partite, weighted simplicial complex such that all the links of $X$ of dimension $>0$ are connected. Fix $0 \leq k \leq n-1$, if there are $\kappa \geq \lambda > \frac{k}{k+1}$ such that 
$$\bigcup_{\tau \in \Sigma (k-1)} Spec (\Delta_{\tau, 0}^+) \setminus \lbrace 0, \frac{n+1-k}{n-k} \rbrace \subseteq [\lambda, \kappa],$$
then 
\begin{dmath*}
\left\Vert \Delta^+_{k} +\frac{n+1-k}{n-k}   \Delta^-_{k} + (k - (k+1)  \dfrac{\lambda + \kappa}{2})  I  - ( \frac{(n+1-k)^2}{n-k} - (n+1-k)^2 \dfrac{\lambda + \kappa}{2} ) \sum_{j=0}^n \Delta^-_{(k,j)}  \right\Vert \leq (k+1) \dfrac{\kappa - \lambda}{2} ,
\end{dmath*}
where $\Vert. \Vert$ denotes the operator norm.
\end{theorem}

\begin{proof}
Let $\phi \in C^{k} (X,\mathbb{R} )$, then for every $\tau \in \Sigma (k-1)$, we have that the projection of $\phi_\tau$ on $C^{0} (X_\tau, \mathbb{R})_{nt}$ is
$$\phi_\tau - (n+1-k) \sum_{j=0}^n \Delta^-_{\tau,(0,j)} \phi_\tau = (I -  (n+1-k) \sum_{j=0}^n \Delta^-_{\tau,(0,j)}) \phi_\tau  .$$
Therefore,
\begin{dmath*}
{\left\langle \Delta^+_\tau  \left(I -  (n+1-k) \sum_{j=0}^n \Delta^-_{\tau,(0,j)} \right) \phi_\tau, \phi_\tau \right\rangle} \geq \lambda \left\Vert \left( I -  (n+1-k) \sum_{j=0}^n \Delta^-_{\tau,(0,j)} \right) \phi_\tau \right\Vert^2 = \lambda \left( \Vert \phi_\tau \Vert^2 - (n+1-k)^2 \sum_{j=0}^n \Vert \Delta^-_{\tau,(0,j)} \phi_\tau \Vert^2 \right).
\end{dmath*}
Similarly,
\begin{dmath*}
 \kappa \left( \Vert \phi_\tau \Vert^2 - (n+1-k)^2 \sum_{j=0}^n \Vert \Delta^-_{\tau,(0,j)} \phi_\tau \Vert^2 \right) \\
 \geq \left\langle \Delta^+_\tau  \left(I -  (n+1-k) \sum_{j=0}^n \Delta^-_{\tau,(0,j)} \right) \phi_\tau, \phi_\tau \right\rangle .
 \end{dmath*}
From the fact that $\Delta^-_{\tau,0}$ is the projection on the constant functions on $X_\tau$ we get that 
$$(n+1-k) \sum_{j=0}^n \Delta^-_{\tau,(0,j)} - \Delta^-_{\tau,0},$$
is the projection of the eigenfunctions with eigenvalue $\frac{n+1-k}{n-k}$. Therefore,
$$\Delta^+_{\tau,0} \left( (n+1-k) \sum_{j=0}^n \Delta^-_{\tau,(0,j)} \right)=\frac{n+1-k}{n-k} \left((n+1-k) \sum_{j=0}^n \Delta^-_{\tau,(0,j)} - \Delta^-_{\tau,0} \right) ,$$
which yields
\begin{dmath*}
\Delta^+_{\tau,0} \left(I - (n+1-k) \sum_{j=0}^n \Delta^-_{\tau,(0,j)} \right) = \Delta^+_{\tau,0} +\frac{n+1-k}{n-k}   \Delta^-_{\tau,0}  - \frac{(n+1-k)^2}{n-k} \sum_{j=0}^n \Delta^-_{\tau,(0,j)} .
\end{dmath*}
Therefore we have that
\begin{dmath*}
 \kappa \left( \Vert \phi_\tau \Vert^2 - (n+1-k)^2 \sum_{j=0}^n \Vert \Delta^-_{\tau,(0,j)} \phi_\tau \Vert^2 \right) \geq \\
 \left\langle \left( \Delta^+_{\tau,0} +\frac{n+1-k}{n-k}   \Delta^-_{\tau,0}  - \frac{(n+1-k)^2}{n-k} \sum_{j=0}^n \Delta^-_{\tau,(0,j)} \right) \phi_\tau, \phi_\tau \right\rangle \geq \\
 \lambda \left( \Vert \phi_\tau \Vert^2 - (n+1-k)^2 \sum_{j=0}^n \Vert \Delta^-_{\tau,(0,j)} \phi_\tau \Vert^2 \right). 
 \end{dmath*}
 Summing the above inequalities on all $\tau \in \Sigma (k-1)$ and using the equalities:
 $$(k+1)! \Vert \phi \Vert^2 = \sum_{\tau \in \Sigma(k-1)} \Vert \phi_\tau \Vert^2 ,$$
$$k! \left\langle\Delta^-_k \phi , \phi \right\rangle= \sum_{\tau \in \Sigma(k-1)} \left\langle \Delta^-_{\tau,0} \phi_\tau , \phi_\tau \right\rangle ,$$
$$  k! \langle \Delta^+_k \phi ,  \phi \rangle + k!k \Vert \phi \Vert^2 =  \sum_{\tau \in \Sigma (k-1)} \langle \Delta^+_{\tau,0} \phi_\tau, \phi_\tau \rangle ,$$
$$k! \langle \Delta^-_{(k,j)} \phi, \phi \rangle = \sum_{\tau \in \Sigma (k-1)} \langle \Delta^-_{\tau, (0,j)} \phi_\tau, \phi_\tau \rangle ,$$ 
(see Lemma \ref{LocalizNorm1}, Corollary \ref{LocalizNorm3} and Proposition \ref{partite localization of Delta^-} ), yields (after dividing by $k!$):
\begin{dmath*}
 \kappa \left\langle \left( (k+1) I -  (n+1-k)^2 \sum_{j=0}^n  \Delta^-_{(k,j)} \right) \phi, \phi \right\rangle  \geq 
 \left\langle \left( \Delta^+_{k} +k I +  \frac{n+1-k}{n-k}   \Delta^-_{k}  - \frac{(n+1-k)^2}{n-k} \sum_{j=0}^n \Delta^-_{(k,j)} \right) \phi, \phi \right\rangle \geq 
 \lambda \left\langle \left( (k+1) I -  (n+1-k)^2 \sum_{j=0}^n  \Delta^-_{(k,j)} \right) \phi, \phi \right\rangle . 
 \end{dmath*}
Subtracting 
$$ \dfrac{\lambda + \kappa}{2}  \left\langle \left( (k+1) I -  (n+1-k)^2 \sum_{j=0}^n  \Delta^-_{(k,j)} \right) \phi, \phi \right\rangle ,$$
from the above inequality yields
\begin{dmath*}
\left\vert \left\langle \left( \Delta^+_{k} +\left( k - (k+1)  \dfrac{\lambda + \kappa}{2} \right)  I +  \frac{n+1-k}{n-k}   \Delta^-_{k}  \\
- \left( \frac{(n+1-k)^2}{n-k} - (n+1-k)^2 \dfrac{\lambda + \kappa}{2} \right) \sum_{j=0}^n \Delta^-_{(k,j)} \right) \phi, \phi \right\rangle \right\vert \\
\leq \dfrac{\kappa - \lambda}{2} \left\langle \left( (k+1) I -  (n+1-k)^2 \sum_{j=0}^n  \Delta^-_{(k,j)} \right) \phi, \phi \right\rangle .
\end{dmath*}
This in turn yields 
\begin{dmath*}
\left\Vert \Delta^+_{k} +\frac{n+1-k}{n-k}   \Delta^-_{k} + \left( k - (k+1)  \dfrac{\lambda + \kappa}{2} \right)  I  - \left( \frac{(n+1-k)^2}{n-k} - (n+1-k)^2 \dfrac{\lambda + \kappa}{2} \right) \sum_{j=0}^n \Delta^-_{(k,j)}  \right\Vert \leq (k+1) \dfrac{\kappa - \lambda}{2} .
\end{dmath*}
\end{proof}
 
 \subsection{Restriction}

\begin{definition}
For $\phi \in C^k (X, \mathbb{R})$ and $\tau \in \Sigma (l)$ s.t. $k+l+1 \leq n$, the restriction of $\phi$ to $X_\tau$  is a function $\phi^\tau \in C^k (X_\tau, \mathbb{R})$ defined as follows: 
$$ \forall \sigma \in \Sigma_\tau (k), \phi^\tau (\sigma) = \phi (\sigma).$$
\end{definition}

For $\phi, \psi \in C^k (X, \mathbb{R} )$, one can compute $\langle \phi , \psi \rangle$ using all the restrictions of the form $\phi^\tau, \psi^\tau$. This is described in the following lemma:

\begin{lemma}
\label{restNorm1}
For every $0 \leq k \leq n-1$ let $\phi, \psi \in C^k (X, \mathbb{R} )$ and $0 \leq l \leq n-k-1$. Then
$$  \langle \phi , \psi  \rangle = \sum_{\tau \in \Sigma (l)} \langle \phi^\tau, \psi^\tau \rangle.$$
\end{lemma}

\begin{proof}
\begin{align*}
 \sum_{\tau \in \Sigma (l)} \langle \phi^\tau, \psi^\tau \rangle = \sum_{\tau \in \Sigma (l)} \sum_{\sigma \in \Sigma_\tau (k)} \dfrac{m_\tau (\sigma)}{(k+1)!}  \phi^\tau (\sigma)\psi^\tau (\sigma) = \\
 \sum_{\tau \in \Sigma (l)} \dfrac{1}{(k+1)!  }\sum_{\sigma \in \Sigma_\tau (k)} m (\tau \sigma) \phi (\sigma)\psi (\sigma) = \\
  \sum_{\tau \in \Sigma (l)} \dfrac{1}{(k+1)! }\sum_{\gamma \in \Sigma (l+k+1),\tau \subset \gamma} \dfrac{(k+1)!}{(l+k+2)! } m (\gamma) \phi (\gamma - \tau)\psi (\gamma - \tau),
\end{align*}
where $\gamma - \tau$ means deleting the vertices of $\tau$ from $\gamma$. Changing the order of summation gives
\begin{align*}
\sum_{\gamma \in \Sigma (l+k+1)} \dfrac{m(\gamma)}{(l+k+2)!  }\sum_{\tau \in \Sigma (l),\tau \subset \gamma}  \phi (\gamma - \tau) \psi (\gamma - \tau)  = \\
\sum_{\gamma \in \Sigma (l+k+1)} \dfrac{m(\gamma)}{(l+k+2)!  }\sum_{\sigma \in \Sigma (k),\sigma \subset \gamma} \dfrac{(l+1)!}{(k+1)!} \phi (\sigma)\psi (\sigma) = \\
 \sum_{\sigma \in \Sigma (k)} \dfrac{(l+1)! \phi (\sigma)\psi (\sigma)}{(l+k+2)!(k+1)! } \sum_{\gamma \in \Sigma (l+k+1),\sigma \subset \gamma} m(\gamma) 
 \end{align*}
 Recall that by Corollary \ref{weight in l dim simplices} we have that 
 $$\sum_{\gamma \in \Sigma (l+k+1),\sigma \subset \gamma} m(\gamma)  = \dfrac{(l+k+2)!}{(l+1)!}  m (\sigma ) .$$
 Therefore we get
\begin{dmath*}
\sum_{\sigma \in \Sigma (k)} \dfrac{(l+1)! \phi (\sigma) \psi (\sigma)}{(l+k+2)!(k+1)! } \sum_{\gamma \in \Sigma (l+k+1),\sigma \subset \gamma} m(\gamma)  = 
 \sum_{\sigma \in \Sigma (k)} \dfrac{ m (\sigma ) }{(k+1)! } \phi (\sigma) \psi (\sigma) = \langle \phi, \psi \rangle .
\end{dmath*}
\end{proof}

\begin{lemma}
\label{restNorm2}
Assume that $X$ is of dimension $>1$. Let $\phi, \psi \in C^0 (X,\mathbb{R})$ and $0 \leq l \leq n-1$, then
$$  \langle d \phi, d \psi  \rangle = \sum_{\tau \in \Sigma (l)} \langle d_\tau \phi^\tau , d_\tau \psi^\tau \rangle,$$
where $d_\tau$ is the restriction of $d$ to the link of $\tau$.
\end{lemma}

\begin{proof}
Note that 
$$ \forall (v_0,v_1) \in \Sigma_\tau (1), d_\tau \phi^\tau ((v_0,v_1)) = \phi (v_0)-\phi (v_1) = d \phi ((v_0,v_1)) = (d \phi)^\tau ((v_0,v_1)),  $$
and similarly
$$ \forall (v_0,v_1) \in \Sigma_\tau (1), d_\tau \psi^\tau ((v_0,v_1)) = (d \phi)^\tau ((v_0,v_1)) .$$
Therefore $d_\tau (\phi^\tau) = (d \phi)^\tau, d_\tau (\psi^\tau) = (d \psi)^\tau$  and the lemma follows from the previous one. 
\end{proof}

\subsection{Connectivity of links}

Throughout this paper, we will assume that the $1$-skeleton of $X$ and the $1$-skeletons of all the links of $X$ of dimension $>0$ are connected. We show that this implies that $X$ has strong connectivity properties, namely we shall show that $X$ is gallery connected (see definition below).

\begin{definition}
A pure $n$-dimensional simplicial complex is called gallery connected, if for every two vertices $u ,v \in X^{(0)}$ there is a sequence of simplexes $\sigma_0 ,...,\sigma_l \in X^{(n)}$ such that $u \in \sigma_0, v \in \sigma_l$ and for every $0 \leq i \leq l-1$, we have that $\sigma_i \cap \sigma_{i+1} \in X^{(n-1)}$.
\end{definition}

\begin{proposition}
\label{gallery connectedness}
Let $X$ be a pure $n$-dimensional simplicial complex, such that the $1$-skeleton of $X$ is connected. If the $1$-skeletons of all the links of $X$ of dimension $>0$ are connected, then $X$ is gallery connected. 
\end{proposition}

\begin{proof}
We shall prove the by induction on $n$. For $n=1$, $X$ is a graph, therefore, $X$ being gallery connected is the same as $X$ being connected. Assume the proposition holds for $n-1$. Let  $X$ be a pure $n$-dimensional simplicial complex, such that the $1$-skeleton of $X$ is connected. If the $1$-skeletons of all the links of $X$ of dimension $>0$ are connected. Then the $(n-1)$-skeleton of $X$ is a pure $(n-1)$-simplicial complex, such that the $1$-skeletons of all the links are connected. Therefore, by the induction assumption, for every $u,v \in  X^{(0)}$, there are $\tau_0,...,\tau_l \in X^{(n-1)}$  such that $u \in \tau_0, v \in \tau_l$ and for every $0 \leq i \leq l-1$, $\tau_i \cap \tau_{i+1} \in X^{(n-2)}$. $X$ is pure $n$-dimensional, therefore we can take $\sigma_i \in X^{(n)}$ such that for every $0 \leq i \leq l$, $\tau_i \subset \sigma_i$. If $l=0$ there is nothing to prove. Assume $l>0$,  to finish, we shall show that for every $0 \leq i \leq l-1$ there is a gallery connecting $\sigma_i$ and $\sigma_{i+1}$ (and therefore one can take a concatenation of those galleries). Fix $0 \leq i \leq l-1$. Denote $\eta = \tau_i \cap \tau_{i+1} \in X^{(n-2)}, v' = \tau_i \setminus \eta, v'' = \tau_{i+1} \setminus \eta$. By our assumptions $X_\eta$ is connected, therefore there are $v_1,...,v_k \in X_\eta^{(0)}$ such that 
$$\lbrace v' , v_1 \rbrace, \lbrace v_1,v_2 \rbrace,...,\lbrace v_k,v'' \rbrace \in X_\eta^{(1)} .$$
Denote 
$$\sigma_0' = \eta \cup \lbrace v', v_1 \rbrace, \sigma_1' = \eta \cup \lbrace v_1, v_2 \rbrace, ...,\sigma_k' =  \eta \cup \lbrace v_k, v'' \rbrace .$$
Note that $\sigma_0',...,\sigma_k' \in X^{(n)}$ and that 
$$\forall 0 \leq i \leq k-1, \sigma_i' \cap \sigma_{i+1}' = \eta \cup \lbrace v_i \rbrace \in X^{(n-1)}.$$
Also note that
$$\tau_i \subseteq \sigma_i \cap  \sigma_0',  \tau_{i+1} \subseteq \sigma_{i+1} \cap  \sigma_k'   .$$
Therefore there is a gallery connecting $\sigma_i$ and $\sigma_{i+1}$ and we are done.

\end{proof}

\section{Spectral gaps of links}

In this section we will show that a large spectral gap on the graph Laplacian (i.e., the upper Laplacian) on all the $1$-dimensional links induces spectral gaps in all the (weighted) graph Laplacians in all the other links (including the $1$-skeleton of $X$ which is the link of the empty set).

We shall show that spectral gaps of the graph Laplacians ``trickle down'' through links of simplices of different dimension. Specifically, we shall show the following:
\begin{lemma}
\label{SpectralGapDescent1}
Let $X$ as before, i.e., a pure $n$-dimensional weighted simplicial complex such that all the links of $X$ of dimension $>0$ are connected. Also, assume that $n>1$. For $0 \leq k \leq n-2$, if there are $\kappa \geq \lambda >0$ such that
$$\bigcup_{\sigma \in \Sigma (k)} Spec (\Delta_{\sigma, 0}^+) \setminus \lbrace 0 \rbrace \subseteq [\lambda, \kappa],$$
then
$$\bigcup_{\tau \in \Sigma (k-1)} Spec (\Delta_{\tau, 0}^+) \setminus \lbrace 0 \rbrace  \subseteq \left[2 - \dfrac{1}{\lambda}, 2 - \dfrac{1}{\kappa} \right].$$ 
 
\end{lemma}

\begin{proof}

Fix some $\tau \in \Sigma (k-1)$. First note that
$$\bigcup_{v \in \Sigma_\tau (0)} Spec (\Delta_{\tau v, 0}^+) \setminus \lbrace 0 \rbrace \subseteq \bigcup_{\sigma \in \Sigma (k)} Spec (\Delta_{\sigma, 0}^+) \setminus \lbrace 0 \rbrace \subseteq [\lambda, \kappa].$$
For every $v \in \Sigma_\tau (0)$ and recall that $\Delta_{\tau v}^- \phi^v$ is the projection of $\phi^v$ to the space of constant maps on $X_{\tau v}$. Denote by $( \phi^v)^1$ the orthogonal compliment of that projection. \\
Since $X_{\tau v}$ is connected for every $v \in \Sigma_\tau (0)$, the kernel of $\Delta^+_{\tau v}$ is the space of constant maps. Therefore for every  $v \in \Sigma_\tau (0)$ we have that
$$\kappa \Vert (\phi^v)^1 \Vert^2 \geq \Vert d_{\tau v} \phi^v \Vert^2 \geq \lambda \Vert (\phi^v)^1 \Vert^2 .$$

Take $\phi \in  C^0 (X_\tau, \mathbb{R})$ to be a non constant eigenfunction of $\Delta^+_\tau$ with the eigenvalue $\mu >0$ (recall that $X_\tau$ is connected so the kernel of $\Delta^+_\tau$ is the space of constant functions) , i.e., 
$$\Delta_\tau^+ \phi (u) = \mu \phi (u).$$ 
By Lemma \ref{restNorm2} we have
$$\mu \Vert \phi \Vert^2 =  \Vert d_\tau \phi \Vert^2 = \sum_{v \in \Sigma_\tau (0)} \Vert d_{\tau v} \phi^v \Vert^2.$$
Combined with the above inequalities this yields:
\begin{equation}
\label{ineq}  
\kappa \sum_{v \in \Sigma_\tau (0)}  \Vert (\phi^v)^1 \Vert^2 \geq  \mu \Vert \phi \Vert^2 \geq \lambda \sum_{v \in \Sigma_\tau (0)}  \Vert (\phi^v)^1 \Vert^2.\end{equation}
Next, we shall compute $\sum_{v \in \Sigma_\tau (0)}  \Vert (\phi^v)^1 \Vert^2 $. Note that
$$\Vert (\phi^v)^1 \Vert^2 = \Vert (\phi^v) \Vert^2 - \Vert \Delta_{\tau v}^- \phi^v \Vert^2.$$  
By Lemma \ref{restNorm1} we have that 
$$\sum_{v \in \Sigma_\tau (0)}  \Vert (\phi^v) \Vert^2 = \Vert \phi \Vert^2,$$
and therefore we need only to compute $\sum_{v \in \Sigma_\tau (0)}  \Vert \Delta_{\tau v}^- \phi^v \Vert^2 $. First, let us write $\Delta_{\tau v}^- \phi^v$ explicitly: 
$$\Delta_{\tau v}^- \phi^v \equiv \dfrac{1}{m_{\tau v} (\emptyset )} \sum_{u \in \Sigma_{\tau v} (0) } m_{\tau v} (u) \phi^v (u) = \dfrac{1}{m_{\tau} (v )} \sum_{(v,u)  \in \Sigma_{\tau} (1) } m_{\tau} ((v,u)) \phi (u).$$
Notice that since $\Delta_\tau^+ \phi = \mu \phi$, we get 
$$\mu \phi (v) = \Delta_\tau^+ \phi (v) = \phi (v) - \dfrac{1}{m_\tau (v)} \sum_{(v,u) \in \Sigma_\tau (1)} m_\tau((v,u)) \phi (u) = \phi (v)-  \Delta_{\tau v}^- \phi^v.$$
Therefore
$$\Delta_{\tau v}^- \phi^v = (1 - \mu ) \phi (v).$$
This yields
$$\sum_{v \in \Sigma_\tau (0)}  \Vert \Delta_{\tau v}^- \phi^v \Vert^2 = \sum_{v \in \Sigma_\tau (0)} m_\tau (v) (1- \mu)^2 \phi (v)^2 = (1- \mu)^2 \Vert \phi \Vert^2.$$ 
Therefore
$$\sum_{v \in \Sigma_\tau (0)}  \Vert (\phi^v)^1 \Vert^2 = \sum_{v \in \Sigma_\tau (0)} \Vert (\phi^v) \Vert^2- \sum_{v \in \Sigma_\tau (0)}  \Vert \Delta_{\tau v}^- \phi^v \Vert^2 = \Vert \phi \Vert^2 (1 - (1- \mu)^2) =\Vert \phi \Vert^2 \mu (2-\mu) .$$
Combine with the inequality in \eqref{ineq} to get
$$ \kappa  \Vert \phi \Vert^2 \mu (2-\mu) \geq  \mu \Vert \phi \Vert^2  \geq \lambda \Vert \phi \Vert^2 \mu (2-\mu) .$$
Dividing by $\Vert \phi \Vert^2 \mu$ yields
$$ \kappa   (2-\mu) \geq  1    \geq \lambda  (2-\mu) .$$
And this in turns yields 
$$ 2- \dfrac{1}{\kappa} \geq \mu \geq 2- \dfrac{1}{\lambda}.$$
Since $\mu$ was any positive eigenvalue of $\Delta_{\tau,0}^+$ we get that
$$Spec (\Delta_{\tau, 0}^+) \setminus \lbrace 0 \rbrace  \subseteq \left[2 - \dfrac{1}{\lambda}, 2 - \dfrac{1}{\kappa} \right].$$

\end{proof}
Our next step is to iterate the above lemma. Consider the function $f(x) = 2 - \frac{1}{x}$. One can easily verify that this function is strictly monotone increasing and well defined on $(0, \infty)$. Denote $f^2 = f \circ f, f^j = f \circ ... \circ f$. Simple calculations show that
$$f^j (x) = \dfrac{(j+1)x - j}{j x - (j-1)}$$
and therefore 
$$\forall m \in \mathbb{N}, f \left( \dfrac{m}{m+1} \right) = \dfrac{m-1}{m}, f(1)=1,$$
and
$$ \forall a >1, \lbrace f^j (a) \rbrace_{j \in \mathbb{N}} \text{ is a decreasing sequence and } \lim_{j \rightarrow \infty}  f^j (a) = 1.$$
The next corollary implies Theorem \ref{Intro-theorem1} of the introduction (for the homogeneous weight): 

\begin{corollary}
\label{SpectralGapDescent2}
Let $X$ be as in the lemma and $f$ as above. Assume that there are $\kappa \geq \lambda >\frac{n-1}{n}$ such that
$$\bigcup_{\sigma \in \Sigma (n-2)} Spec (\Delta_{\sigma, 0}^+) \setminus \lbrace 0 \rbrace \subseteq [\lambda, \kappa],$$
then for every $-1 \leq k \leq n-3$ we have
$$\bigcup_{\tau \in \Sigma (k)} Spec (\Delta_{\tau, 0}^+) \setminus \lbrace 0 \rbrace  \subseteq \left[ f^{n-k-2} (\lambda) , f^{n-k-2} (\kappa) \right] \subseteq \left( \dfrac{k+1}{k+2} , \dfrac{n-k}{n-k-1} \right].$$ 
\end{corollary}

\begin{proof}
The proof is a straightforward induction using Lemma \ref{SpectralGapDescent1}. One only needs to verify that for every $-1 \leq k \leq n-3$ we have $f^{n-k-2} (\lambda) >0$, but this is guaranteed by the condition $\lambda >\frac{n-1}{n}$ which implies that $f^{n-k-2} (\lambda) > \frac{k+1}{k+2}$. The upper bound on the spectrum stems from the fact that $\kappa \leq 2$ and therefore $f^{n-k-2} (\kappa) \leq f^{n-k-2} (2)= \frac{n-k}{n-k-1}$.
\end{proof}

\section{Spectral gap of links in partite complexes}

In this section we derive a result similar to Corollary \ref{SpectralGapDescent2} in the case where $X$ is a partite complex. Motivated by the discussion in \ref{Localization in partite complexes subsection} above, we want to bound the eigenvalues of the links of partite complexes from above, excluding the largest eigenvalue - see Corollary \ref{spectral descent in partite case - corollary} below. In the following two propositions, we start with an analysis of the structure of the eigenfunction of the largest eigenvalue of the links in the partite case:

\begin{proposition}
\label{kappa = 2 bound}
Let $X$  be a pure $n$-dimensional weighted simplicial complex such that all the links of $X$ of dimension $>0$ are connected. For every $-1 \leq k \leq n-3$, every $\tau \in X^{(k)}$ and every $\phi \in C^0 (X_\tau, \mathbb{R})$, 
$$\Delta_{\tau, 0}^+ \phi = \dfrac{n-k}{n-k-3} \phi \Rightarrow \forall \sigma \in X_\tau^{(n-k-1)}, \Delta_{\sigma, 0}^+ \phi^\sigma = 2 \phi^\sigma .$$
\end{proposition}

\begin{proof}
Let $-1 \leq k \leq n-2$, $\tau \in X^{(k)}$ and $\phi \in C^0 (X_\tau, \mathbb{R})$ such that $\Delta_0^+ \phi = \mu \phi$.  Assume there is a single $v \in X_\tau^{(0)}$ such that (in the notations of the proof of Lemma \ref{SpecGapLocalToGlobal1} )
$$ \dfrac{n-k-1}{n-k-2} \Vert (\phi^v)^1 \Vert^2 > \Vert d_{\tau v} \phi^v \Vert^2.$$
We note that $\phi^v \in C^0 (X_{\tau v}, \mathbb{R})$ and therefore by Corollary \ref{SpectralGapDescent2}, for any other $v \in X_\tau^{(0)}$, we have 
$$ \dfrac{n-k-1}{n-k-2} \Vert (\phi^v)^1 \Vert^2 \geq \Vert d_{\tau v} \phi^v \Vert^2.$$
Therefore we can repeat the proof of Lemma \ref{SpecGapLocalToGlobal1}, with strict inequalities. Namely, instead of inequality \eqref{ineq}, we can take
$$ \dfrac{n-k-1}{n-k-2} \sum_{v \in \Sigma_\tau (0)}  \Vert (\phi^v)^1 \Vert^2 > \mu \Vert \phi \Vert^2,$$
and complete the rest of the proof with strict inequalities and get
$$2 - \dfrac{1}{\frac{n-k-1}{n-k-2}} > \mu ,$$
which yields 
$$\dfrac{n-k}{n-k-1} > \mu.$$
Therefore 
$$\Delta_{\tau,0}^+ \phi = \dfrac{n-k}{n-k-1}  \phi \Rightarrow \forall v \in X_\tau^{(0)}, \Delta_{\tau v,0}^+ \phi^v = \dfrac{n-k-1}{n-k-2}  \phi^v.$$
Finish by induction on $k$, starting with $k=n-3$ and descending.
\end{proof}

The proposition below states that the $1$-skeleton of a partite complex has a known upper bound on the spectrum (this is a generalization of the fact the a bipartite graph has an eigenvalue $2$):  
 
\begin{proposition}
\label{upper bound of the spectrum in partite complexes - prop} 
Let $X$  be a pure $n$-dimensional weighted simplicial complex such that all the links of $X$ of dimension $>0$ are connected. If $X$ is $(n+1)$-partite then $\frac{n+1}{n} \in Spec (\Delta_0^+)$ and the space of eigenfunctions of the eigenvalue $\frac{n+1}{n}$ is spanned by the functions $\varphi_i$, $0 \leq i \leq n$ defined as
$$\varphi_i (u) = \begin{cases}
n & u \in S_i \\
-1 & \text{otherwise}
\end{cases}.$$
\end{proposition}

\begin{proof}
First we verify that each $\varphi_i$ defined above is indeed an eigenfunction of the eigenvalue $\frac{n+1}{n}$. We check the following cases:
\begin{enumerate}
\item In the case $u \in S_i$, we have that
\begin{dmath*}
\Delta^+_0 \varphi_i (u) = \varphi_i (u) - \sum_{v \in X^{(0)}, (u,v) \in \Sigma(1)} \dfrac{m((u,v))}{m(u)} \varphi_i (v) = 
n - \sum_{v \in X^{(0)}, (u,v) \in \Sigma(1)} \dfrac{m((u,v))}{m(u)} (-1) =  
n + \sum_{v \in X^{(0)}, (u,v) \in \Sigma(1)} \dfrac{m((u,v))}{m(u)} = n+1 = \dfrac{n+1}{n} \varphi_i (u) . 
\end{dmath*}
\item In the case where $u \notin  S_i$, we have that 
\begin{dmath*}
\Delta^+_0 \varphi_i (u) = -1  - \sum_{v \in X^{(0)}, (u,v) \in \Sigma(1)} \dfrac{m((u,v))}{m(u)} \varphi_i (v) = 
-1 - \sum_{v \in X^{(0)} \setminus S_i , (u,v) \in \Sigma(1)} \dfrac{m((u,v))}{m(u)} (-1) - \sum_{v \in S_i , (u,v) \in \Sigma(1)} \dfrac{m((u,v))}{m(u)} n.
\end{dmath*}

Recall that by Proposition \ref{weight in n dim simplices} and by the fact that $X$ is pure $n$-dimensional and $(n+1)$-partite, we have that
\begin{dmath*}
m(u) = n! \sum_{\sigma \in X^{(n)}, u \subset \sigma } m(\sigma) = 
n! \sum_{v \in S_i, (u,v) \in \Sigma (1) } \sum_{\sigma \in X^{(n)}, \lbrace u, v \rbrace \subset \sigma } m(\sigma) =
n \sum_{v \in S_i, (u,v) \in \Sigma (1) } m((u,v)) .
\end{dmath*}
Similarly, 
$$ (n-1) m(u) = n \sum_{v \in X^{(0)} \setminus S_i, (u,v) \in \Sigma (1) } m((u,v)) .$$
Therefore we get
\begin{dmath*}
\Delta^+_0 \varphi_i (u) = -1 + \dfrac{n-1}{n} - 1  = - \dfrac{n+1}{n} =\dfrac{n+1}{n} \varphi_i (u) .
\end{dmath*}
\end{enumerate}
Next, we will prove that $\varphi_i$ span the space of eigenfunctions with eigenvalue $\frac{n+1}{n}$. 
For $n=1$, this is the classical argument for bipartite graphs repeated here for the convenience of the reader. Let $\phi \in C^0 (X,\mathbb{R} )$ such that $\Delta_0^+ \phi = 2 \phi$ and $X$ is a bipartite graph. There is $u_0 \in X^{(0)}$ such that $\forall v \in X^{(0)}, \vert \phi (u_0 ) \vert \geq \vert \phi (v ) \vert$. Without loss of generality $u_0 \in S_0$. One can always normalize $\phi$ such that $\phi (u_0) =1$ and for every other $v \in X^{(0)}$, $\vert \phi (v) \vert \leq 1$ . Then 
\begin{dmath*}
2 = \Delta \phi (u_0) = 1 - \sum_{v \in X^{(0)}, (u_0,v) \in \Sigma(1)} \dfrac{m((u_0,v))}{m(u_0)} \phi (v) = 
1 - \sum_{v \in X^{(0)}_1, (u_0,v) \in \Sigma(1)} \dfrac{m((u_0,v))}{m(u_0)} \phi (v)
\end{dmath*}
Therefore 
$$ \sum_{v \in X^{(0)}_1, (u_0,v) \in \Sigma(1)} \dfrac{m((u_0,v))}{m(u_0)} \phi (v) = -1.$$
Note that $ \sum_{v \in X^{(0)}_1, (u_0,v) \in \Sigma(1)} \dfrac{m((u_0,v))}{m(u_0)}  =1 $ and $\forall v, \phi (v) \geq -1$ and therefore we get that for every $v \in X^{(0)}_1$ with $(u_0,v) \in \Sigma(1)$ we get $\phi (v) = -1$. By the same considerations, for every $v \in X^{(0)}$ with $\phi (v) =-1$, we have 
$$u \in X^{(0)}, (v,u) \in \Sigma (1) \Rightarrow \phi (u) =1.$$ 
Therefore by iterating this argument and using the fact that the graph is connected, we get that 
$$\phi (u) = \begin{cases}
 1 & u \in S_0 \\
 -1 & u \in X^{(0)}_1
\end{cases},$$
and that is exactly $\varphi_0$ in the case $n=1$. Assume that $n>1$. \\
First, for every $0 \leq i \leq n$, note that $\chi_{S_i} = \frac{1}{n+1} (\varphi_i + \chi_{X^{(0)}}) $ (Recall that $\chi_{X^{(0)}}$ denotes the constant $1$ function and $\chi_{S_i}$ denotes the indicator function of $S_i$ ). 
Therefore every function $\phi$ of the form: 
$$\exists c_0,...,c_n \in \mathbb{R}, \forall u \in S_i, \phi (u) = c_i ,$$
is in the space spanned by the functions $\varphi_i$ and the constant functions. Therefore, for $\phi$ such that $\Delta_0^+ \phi = \frac{n+1}{n} \phi$, it is enough to show that $\phi$ is of the form 
$$\exists c_0,...,c_n \in \mathbb{R}, \forall u \in S_i, \phi (u) = c_i .$$
Let $\phi \in C^0 (X,\mathbb{R})$ such that $\Delta_0^+ \phi = \frac{n+1}{n} \phi$. Fix $0 \leq i \leq n$ and $u' \in S_i$. By Proposition \ref{gallery connectedness}, $X$ is gallery connected so for every $u \in S_i$ there is a gallery $\sigma_0,...,\sigma_l \in X^{(n)}$ connecting $u'$ and $u$. We will show by induction on $l$ that $\phi (u) = \phi (u')$. For $l=0$, $u=u'$ and we are done. Assume the claim is true for $l$. Let $u \in  S_i$ such that the shortest gallery connecting $u'$ and $u$ is  $\sigma_0,...,\sigma_{l+1} \in X^{(n)}$. By the fact that $X$ is $(n+1)$-partite, there is $u'' \in \sigma_l  \cap S_i$ therefore $u'', u$ are both in the link of  $\sigma_l  \cap \sigma_{l+1} \in X^{(n-1)}$. Since $n>1$, $\sigma_l  \cap \sigma_{l+1}$ is of dimension $>1$, therefore there is a non empty simplex $\tau \in X^{(n-2)}$ such that $\tau \subset \sigma_l  \cap \sigma_{l+1}$. Note that by the $(n+1)$-partite assumption of $X$, we have that the link of $X_\tau$ is a bipartite graph, containing $u''$ and $u$. From Proposition \ref{kappa = 2 bound} we have that 
$$\Delta_0^+ \phi^\tau = 2 \phi^\tau.$$
Therefore, from the case $n=1$, we get that 
$$\phi (u'') = \phi^\tau (u'') = \phi^\tau (u) =\phi (u).$$ 
By our induction assumption, $\phi (u') = \phi (u'') $ and therefore $\phi$ must be of the form stated above and we are done.
\end{proof}

The above proposition indicates that when dealing with an $(n+1)$-partite simplicial complex, one should think of the non-trivial spectrum of $\Delta^+_0$ as $Spec (\Delta^+_0 ) \setminus \lbrace 0 , \frac{n+1}{n} \rbrace$. 
Following this logic, we denote the space of non-trivial functions $C^0 (X,\mathbb{R} )_{nt}$ as 
$$C^0 (X,\mathbb{R} )_{nt} = span \lbrace \chi_{X^{(0)}}, \varphi_0,...,\varphi_n \rbrace^\perp .$$

\begin{proposition}
Let $\chi_{S_i}$ be the indicator function of $S_i$, then
$$C^0 (X,\mathbb{R} )_{nt} = span \lbrace \chi_{S_0},...,\chi_{S_n} \rbrace^\perp .$$
Moreover, for every $\phi \in C^0 (X,\mathbb{R} )$, the projection of $\phi$ on $C^0 (X,\mathbb{R} )_{nt}$ is 
$$\phi - (n+1) \sum_{j=0}^n \Delta^-_{(0,j)} \phi .$$
\end{proposition}

\begin{proof}
As noted in the proof of the proposition above, $\chi_{S_i} = \frac{1}{n+1} (\varphi_i + \chi_{X^{(0)}}) $. Also notice that
$$\chi_{X^{(0)}} = \sum_{i=0}^n \chi_{S_i} ,$$
$$\forall i, \varphi_i = \sum_{j=0}^n \chi_{X^{(0)}_j} + (n-1) \chi_{S_i}.$$
Therefore 
$$ span \lbrace \varphi_0,...,\varphi_n, \chi_{X^{(0)}} \rbrace =  span \lbrace \chi_{S_0},...,\chi_{S_n} \rbrace .$$
Notice that for every $j$, 
$$\Vert \chi_{S_j} \Vert^2 = \sum_{v \in S_j} m(v) = \dfrac{1}{n+1} m (\emptyset ),$$
and for every $\phi \in C^0 (X,\mathbb{R} )$
\begin{dmath*}
\langle \phi , \chi_{S_j} \rangle \chi_{S_j} = \left( \sum_{v \in S_j} \phi (v) \right) \chi_{S_j}  =  m(\emptyset ) \Delta^-_{(0,j)} \phi.
\end{dmath*}
Therefore, for every $\phi \in C^0 (X,\mathbb{R} )$, the projection of $\phi$ on $C^0 (X,\mathbb{R} )_{nt}$ is 
\begin{dmath*}
\sum_{j=0}^n \dfrac{1}{\Vert \chi_{S_j}  \Vert^2} \langle \phi, \chi_{S_j}  \rangle \chi_{S_j} = \dfrac{n+1}{m(\emptyset )} \sum_{j=0}^n  m(\emptyset ) \Delta^-{(0,j)} \phi = (n+1) \sum_{j=0}^n \Delta^-_{(0,j)} \phi .
\end{dmath*}
\end{proof}

Next, we have a technical tool to calculate to norm and Laplacian of functions in $C^0 (X, \mathbb{R})_{nt}$:

\begin{proposition}
Let $X$  be a pure $n$-dimensional, $(n+1)$-partite, weighted simplicial complex such that all the links of $X$ of dimension $>0$ are connected. Let $\phi \in C^0 (X, \mathbb{R})$. For every $0 \leq i \leq n$, define $\phi_{i} (u) \in C^0 (X, \mathbb{R})$ as follows:
$$\phi_{i} (u) = \begin{cases}
-n \phi (u) & u \in S_i \\
 \phi (u) & \text{otherwise}
\end{cases} .$$
Then 
\begin{enumerate}
\item If $\phi \in C^0 (X, \mathbb{R})_{nt}$, then for every $0 \leq i \leq n$, we have that $\phi_{i} (u) \in C^0 (X, \mathbb{R})_{nt}$. 
\item For every $\phi \in C^0 (X, \mathbb{R})$,
$$\sum_{i=0}^n \Vert \phi_{i} \Vert^2 = (n^2 + n) \Vert \phi \Vert^2 .$$
\item For every $\phi \in C^0 (X, \mathbb{R})$,
$$\sum_{i=0}^n \langle \phi_{i}, \Delta^+_0 \phi_{i} \rangle = \langle \phi , ((n+1)^2 I - (n+1) \Delta^+_0 ) \phi \rangle.$$
\end{enumerate} 
\end{proposition}

\begin{proof}
\begin{enumerate}
\item Let $\phi \in C^0 (X,\mathbb{R} )_{nt} $. Fix $0 \leq i \leq n$. Note that for every $0 \leq j \leq n$, we have that
$$\langle \phi , \chi_{X^{(0)}_j} \rangle = 0  \Rightarrow \langle \phi_{i} , \chi_{X^{(0)}_j} \rangle = 0, $$
and therefore by the above proposition $\phi_{i} (u) \in C^0 (X, \mathbb{R})_{nt}$. 
\item For every $0 \leq i \leq n$ we have that
\begin{dmath*}
\Vert \phi_{i} \Vert^2 = \sum_{u \in S_i} m(u) n^2 \phi (u)^2 +  \sum_{u \in X^{(0)} \setminus S_i} m(u) \phi (u)^2.
\end{dmath*}
Therefore
\begin{dmath*}
\sum_{i=0}^n \Vert \phi_{i} \Vert^2 = \sum_{u \in X^{(0)}} m(u) (n^2+n) \phi (u)^2 = (n^2 +n ) \Vert \phi \Vert^2 .
\end{dmath*}
\item For every $0 \leq i \leq n$, we will compute $\Delta^+_0 \phi_{i}$: For $u \in S_i$, we have that
\begin{dmath*}
(\Delta^+_0 \phi_{i}) (u) = -n \phi (u) - \sum_{v \in X^{(0)}, (u,v) \in \Sigma (1)} \dfrac{m((u,v))}{m(u)} \phi (v) = (-n-1) \phi (u) + (\Delta^+_0 \phi) (u). 
\end{dmath*}
For $u \in X^{(0)} \setminus S_i$ we have that
\begin{dmath*}
(\Delta^+_0 \phi_{i}) (u) = \phi (u) - \sum_{v \in X^{(0)} \setminus S_i, (u,v) \in \Sigma (1)} \dfrac{m((u,v))}{m(u)} \phi (v) - \sum_{v \in S_i, (u,v) \in \Sigma (1)} \dfrac{m((u,v))}{m(u)} (-n) \phi (v) = (\Delta^+_0 \phi) (u) + (n+1) \sum_{v \in S_i, (u,v) \in \Sigma (1)} \dfrac{m((u,v))}{m(u)} \phi (v) .
\end{dmath*}
Therefore
\begin{dmath*}
\langle  \phi_{i},\Delta^+_0 \phi_{i} \rangle = \sum_{u \in S_i} m(u) \phi (u) \left( -n (-n-1) \phi (u) -n (\Delta^+_0 \phi) (u) \right) + \sum_{u \in X^{(0)} \setminus S_i} m(u) \phi (u) \left( (\Delta^+_0 \phi) (u) + (n+1) \sum_{v \in S_i, (u,v) \in \Sigma (1)} \dfrac{m((u,v))}{m(u)} \phi (v) \right).
\end{dmath*}
This yields
\begin{dmath*}
\sum_{i=0}^n \langle  \phi_{i},\Delta^+_0 \phi_{i} \rangle = \sum_{u \in X^{(0)}} m(u) \phi (u) \left( n (n+1) \phi (u) -n (\Delta^+_0 \phi) (u) + n (\Delta^+_0 \phi) (u) + (n+1)  \sum_{v \in X^{(0)}, (u,v) \in \Sigma (1)} \dfrac{m((u,v))}{m(u)} \phi (v) \right) = \sum_{u \in X^{(0)}} m(u) \phi (u) \left( (n+1)^2 \phi (u) - (n+1) (\Delta^+_0 \phi) (u)  \right)= \langle \phi , ((n+1)^2 I - (n+1) \Delta^+_0 ) \phi \rangle.
\end{dmath*}
\end{enumerate}

\end{proof}

It is known that for bipartite graph, the spectrum of the Laplacian is symmetric around $1$. For $(n+1)$-partite complexes we have a weaker result that shows that the bounds of the non-trivial spectrum have some symmetry around $1$:

\begin{lemma}
\label{spectral bound in the partite case} 
Let $X$  be a pure $n$-dimensional, $(n+1)$-partite, weighted simplicial complex such that all the links of $X$ of dimension $>0$ are connected. Assume that $X$ is non-trivial, i.e., assume that $X$ has more than a single $n$-dimensional simplex. Denote 
$$\lambda (X) = \min \lbrace \lambda :  \lambda > 0, \exists \phi, \Delta^+_0 \phi = \lambda \phi \rbrace ,$$
$$\kappa (X) = \max \lbrace \lambda :  \lambda < \frac{n+1}{n}, \exists \phi, \Delta^+_0 \phi = \lambda \phi \rbrace .$$ 
Then 
$$1 - \frac{1}{n} (1 -\lambda (X)) \leq \kappa (X) \leq  1 - n (1-\lambda (X)).$$

\end{lemma}

\begin{proof}
Let $\phi \in C^0 (X,\mathbb{R} )_{nt} $ by the eigenfunction of $\kappa (X)$. By the above proposition, for every $0 \leq i \leq n$, $\phi_{i} \in C^0 (X,\mathbb{R} )_{nt} $ and therefore
$$ \langle \phi_{i} , \Delta^+_0 \phi_{i} \rangle \geq \lambda (X) \Vert \phi_{i}  \Vert^2 .$$
Summing on $i$ we get
$$  \sum_{i=0}^n \langle \phi_{i} , \Delta^+_0 \phi_{i} \rangle \geq \lambda (X)  \sum_{i=0}^n \Vert \phi_{i}  \Vert^2 .$$
By the equalities proven in the above proposition, this yields
$$   \langle \phi , ((n+1)^2 I - (n+1) \Delta^+_0 ) \phi \rangle \geq \lambda (X)  (n^2+n) \Vert \phi \Vert^2 .$$
Since we took $\phi$ to be the eigenfunction of $\kappa (X)$, this yields 
$$ ((n+1)^2 - (n+1) \kappa (X)) \Vert \phi \Vert^2 \geq  \lambda (X)  (n^2+n) \Vert \phi \Vert^2 .$$
Therefore 
$$1 +n (1- \lambda (X)) \geq \kappa (X).$$
By the same procedure, when $\phi$ is taken to be the eigenfunction of $\lambda (X)$, we get that
$$ ((n+1)^2 - (n+1) \lambda (X)) \Vert \phi \Vert^2 \leq  \kappa (X)  (n^2+n) \Vert \phi \Vert^2 ,$$
and therefore
$$1 + \dfrac{1}{n} (1- \lambda (X)) \leq \kappa (X) .$$

\end{proof}

The discussion above leads to a version of spectral descent in partite complexes: as explained above, for a partite complex $X$, for every $\tau \in \Sigma (k)$, $X_\tau$ is $(n-k)$-partite, i.e., the Laplacian of its $1$-skeleton always has $\frac{n-k}{n-k-1}$ as an eigenvalue. Therefore an analogue of the two-sided spectral gap is the spectrum without this eigenvalue. Combining the previous Lemma with Corollary \ref{SpectralGapDescent2} yields the following version of spectral descent in partite complexes, which implies Theorem \ref{Intro-theorem2} of the introduction (for the homogeneous weight):

\begin{corollary}
\label{spectral descent in partite case - corollary}
Let $X$ be a pure $n$-dimensional weighted simplicial complex which is partite such that all the links of $X$ of dimension $>0$ are connected. Assume that there is $\lambda >\frac{n-1}{n}$ such that
$$\bigcup_{\sigma \in \Sigma (n-2)} Spec (\Delta_{\sigma, 0}^+) \setminus \lbrace 0 \rbrace \subseteq [\lambda, 2],$$
then for every $-1 \leq k \leq n-3$, we have
$$\bigcup_{\tau \in \Sigma (k)} Spec (\Delta_{\tau, 0}^+) \setminus \lbrace 0, \frac{n-k}{n-k-1} \rbrace  \subseteq \left[ f^{n-k-2} (\lambda) , 1-(n-k)(1-f^{n-k-2} (\lambda)) \right],$$
where $f=2-\frac{1}{x}$. 
\end{corollary}

\section{Spectral gaps of higher Laplacians}

Our results bounding the spectral gaps in the links using the spectral gaps in the $1$-dimensional links yield spectral gaps in the higher Laplacians as a direct result of Garland's method explained above:

\begin{corollary}
\label{SpecGapCor}
Let $X$ be a pure $n$-dimensional weighted simplicial complex such that all the links of $X$  of dimension $>0$ are connected. Also, assume that $n>1$. Denote $f(x) = 2-\frac{1}{x}$ and $f^j$ to be the composition of $f$ with itself $j$ times (where $f^0$ is defined as $f^0 (x) = x$).
If there are $\kappa \geq \lambda > \frac{n-1}{n}$ such that
$$\bigcup_{\tau \in \Sigma (n-2)} Spec (\Delta_{\tau, 0}^+) \setminus \lbrace 0 \rbrace \subseteq [\lambda, \kappa].$$
Then for every $0 \leq k \leq n-1$:
$$ Spec (\Delta_{k}^+) \setminus \lbrace 0 \rbrace \subseteq [(k+1 ) f^{n-1-k} (\lambda) - k, (k+1) f^{n-1-k} (\kappa)- k], $$
$$Spec (\Delta_{k+1}^-) \setminus \lbrace 0 \rbrace \subseteq [(k+1 ) f^{n-1-k} (\lambda) - k, (k+1) f^{n-1-k} (\kappa)- k] .$$ 
\end{corollary}

\begin{proof}
First apply Corollary \ref{SpectralGapDescent2} to get spectral gaps of $\Delta_{\tau, 0}^+$ for every $\tau \in \Sigma (k)$ when $-1 \leq k \leq n-3$ in terms of $f$ and $\lambda, \kappa$ (notice that since $X_\emptyset = X$ this takes care of the case $k=0$ in $3.$ of the theorem). Then apply Corollary \ref{SpecGapLocalToGlobal2} to finish the proof.
\end{proof}

\begin{remark}
As remarked earlier, if $m$ is the homogeneous weight function, then for every $\tau \in \Sigma (n-2)$, $\Delta_{\tau, 0}^+$ is the usual graph Laplacian on the graph $X_\tau$. This means that if one assigns the  homogeneous weight on $X$, then the spectral gap conditions stated in the above theorem are simply spectral gaps conditions of the usual graph Laplacian on each of the $1$-dimensional links. In concrete examples, these spectral gap conditions are usually attainable.
\end{remark}

\begin{corollary}
\label{contraction corollary}
Let $X$ be a pure $n$-dimensional weighted simplicial complex such that all the links of $X$  of dimension $>0$ are connected. Also, assume that $n>1$. Denote $f(x) = 2-\frac{1}{x}$ and $f^j$ to be the composition of $f$ with itself $j$ times (where $f^0$ is defined as $f^0 (x) = x$).
If there are $\kappa \geq \lambda > \frac{n-1}{n}$ such that
$$\bigcup_{\tau \in \Sigma (n-2)} Spec (\Delta_{\tau, 0}^+) \setminus \lbrace 0 \rbrace \subseteq [\lambda, \kappa].$$
Then for every $0 \leq k \leq n-1$,
\begin{align*}
\left\Vert \Delta_k^+ + \dfrac{\lambda + \kappa}{2} \Delta_k^- - (k+1) (\dfrac{\lambda + \kappa}{2} - \dfrac{k}{k+1}) I \right\Vert  \leq \\ \dfrac{k+1}{2} (f^{n-1-k} (\kappa) - f^{n-1-k} (\lambda)).
\end{align*}

\end{corollary}

\begin{proof}
The proof is a straightforward combination of Corollary \ref{SpectralGapDescent2} and Corollary \ref{norm bound - local to global}.
\end{proof}

Similarly, in the case of partite complexes, we get that:

\begin{corollary}
\label{norm bound - n+1 partite case}
Let $X$  be a pure $n$-dimensional, $(n+1)$-partite, weighted simplicial complex such that all the links of $X$ of dimension $>0$ are connected. Denote $f(x) = 2-\frac{1}{x}$ and $f^j$ to be the composition of $f$ with itself $j$ times (where $f^0$ is defined as $f^0 (x) = x$).
If there is $\lambda > \frac{n-1}{n}$ such that
$$\bigcup_{\tau \in \Sigma (n-2)} Spec (\Delta_{\tau, 0}^+) \setminus \lbrace 0 \rbrace \subseteq [\lambda, 2].$$
Then for every $0 \leq k \leq n-1$,
\begin{dmath*}
\left\Vert \Delta^+_{k} +\frac{n+1-k}{n-k}   \Delta^-_{k} - \left(  \dfrac{2+(n-k)(1-f^{n-1-k}(\lambda))}{2} \right)  I  - \left( \frac{(n+1-k)^2}{n-k} - (n+1-k)^2 \dfrac{2+(n-k)(1-f^{n-1-k}(\lambda))}{2} \right) \cdot \sum_{j=0}^n \Delta^-_{(k,j)}  \right\Vert \leq (k+1)(n+1-k) \dfrac{1 - f^{n-1-k}(\lambda)}{2} ,
\end{dmath*}
where $\Vert. \Vert$ denotes the operator norm.
\end{corollary}

\begin{proof}
The proof is a straightforward combination of Corollary \ref{spectral descent in partite case - corollary} with Theorem \ref{Contraction in partite case thm}.
\end{proof}

\bibliographystyle{alpha}
\bibliography{biblI}

\end{document}